\documentclass[11pt,letterpaper]{article}%
\usepackage{amsmath}
\usepackage{amsfonts}
\usepackage{amssymb}
\usepackage{graphicx}
\usepackage{pstricks}
\usepackage{pst-node}
\usepackage{pst-plot}%

\providecommand{\U}[1]{\protect\rule{.1in}{.1in}}

\newcommand{\C}{{\mathbb C}}

\newcommand{\R}{{\mathbb R}}
\newcommand{\Z}{{\mathbb Z}}

\newcommand{\be}{\begin{equation}}
\newcommand{\ee}{\end{equation}}

\newcommand{\ba}{\begin{eqnarray}}
\newcommand{\ea}{\end{eqnarray}}

\renewcommand{\Im}{\operatorname{Im}}

\newtheorem{theorem}{Theorem}
\newtheorem{proposition}[theorem]{Proposition}

\newtheorem{lemma}[theorem]{Lemma}
\newtheorem{definition}[theorem]{Definition}
\newenvironment{proof}[1][Proof]{\noindent\textbf{#1.} }{\ \rule{0.5em}{0.5em}}
\newtheorem{preremark}[theorem]{Remark}
\newenvironment{remark}{\begin{preremark}\rm}{\hfill$\Diamond$\end{preremark}}
\newtheorem{prenotation}[theorem]{Notation}

\numberwithin{equation}{section}
\numberwithin{theorem}{section}
\begin{document}

\title{{Degeneration of K\"ahler structures}\\and half-form quantization of toric varieties}
\author{William D. Kirwin\thanks{current address: Mathematics Institute, University of Cologne, Weyertal 86 -- 90, 50931 Cologne, GERMANY}, Jos\'{e} M. Mour\~{a}o and Jo\~{a}o P. Nunes}
\maketitle

\begin{center}
{Center for Mathematical Analysis, Geometry and Dynamical Systems\\and\\
Department of Mathematics\\ Instituto Superior T\'ecnico\\ Av. Rovisco Pais\\ 1049-001 Lisbon, Portugal}
\end{center}

\bigskip
\begin{abstract}
We study the half-form K\"{a}hler quantization of a smooth symplectic toric
manifold $(X,\omega)$, such that $[ \omega/ 2\pi]- c_{1}(X)/2 \in
H^{2}(X,{\mathbb{Z}} )$ and is nonnegative. We define the half-form corrected
quantization of $(X,\omega)$ to be given by holomorphic sections of a certain
hermitian line bundle $L\to X$ with Chern class $[ \omega/ 2\pi]- c_{1}%
(X)/2$. These sections then correspond to integral points of a ``corrected''
polytope $P_{L}$ with integral vertices. For a suitably translated moment
polytope $P_{X}$ for $(X,\omega)$, we have that $P_{L}\subset P_{X}$ is
obtained from $P_{X}$ by a one-half inward-pointing normal shift along the
boundary.

We use our results on the half-form corrected K\"ahler quantization to motivate a definition of half-form
corrected quantization in the  singular real toric polarization.
Using families of complex structures studied in
\cite{Baier-Florentino-Mourao-Nunes}, which include the degeneration of
K\"{a}hler polarizations to the vertical polarization, we show that, under
this degeneration, the half-form corrected $L^{2}$-normalized monomial holomorphic
sections converge to Dirac-delta-distributional sections supported on the
fibers over the integral points of $P_{L}$, which correspond to corrected Bohr--Sommerfeld fibers.
This result and the limit of the corrected connection, with curvature singularities
along the boundary of $P_X$,  justifies the direct definition we give for the corrected
quantization in the singular real toric polarization. We show that the space of quantum states
for this definition coincides with the space obtained via degeneration of the K\"ahler quantization.

We also show that the BKS pairing
between K\"{a}hler polarizations is not unitary in general. On the other hand, the
unitary connection induced by this pairing is flat.
\end{abstract}

\pagebreak
\tableofcontents

\section{Introduction}

Ever since \'Snyaticki proposed cohomological wave functions to construct the
quantum Hilbert space corresponding to geometric quantization in real
polarizations \cite{Sniatycki}, the question of how to address the case of
real polarizations with singular fibers has resisted full treatment. In
\cite{Hamilton07}, Hamilton proposed the extension of \'Snyaticki's
definition to the case with singular fibers by also considering the higher
cohomology of the same sheaf of polarized smooth sections of the
prequantization bundle. His results show, however, that the formalism will
have to be modified in order to obtain the expected quantization even in the case
of the harmonic oscillator. Indeed, for singularities of elliptic type
(like in the case of toric varieties) Hamilton obtains states corresponding
only to non-singular Bohr--Sommerfeld leaves. In the toric case, these
correspond to interior integral points of the moment polytope. If one doesn't
take into account the half-form correction, however, one expects the quantization to
include \emph{all} states
corresponding to the integral points of the polytope, including those on the
boundary. Only in this way, for the compact case, does one get the same dimension of the space of
quantum states as for the holomorphic polarizations.

In \cite{Baier-Florentino-Mourao-Nunes}, a solution of this problem was
proposed within the context of toric varieties without the half-form correction.
The real polarized sections are defined directly as distributional solutions
of the equations of covariant constancy and can also be obtained by
degenerating appropriately normalized K\"ahler polarized sections. These
normalized holomorphic sections converge, under the degeneration, to Dirac-delta-distributional sections
supported on the Bohr--Sommerfeld fibers which correspond to integral points of the moment polytope, including the ones on the boundary. The corresponding Bohr--Sommerfeld orbits are increasingly singular (lower dimensional) as the codimension of the face of the polytope on which they are
increases.

On the other hand, one would expect these quantum states not to be present in a quantization
in the real toric polarization correctly reproducing the ``vacuum energy shift''
of the harmonic oscillator. We show that this expected behavior of the quantum states
is precisely achieved by our definition of the half-form corrected K\"ahler quantizations
deforming continuously to the real polarization.

An immediate obstacle to defining the half-form quantization in a K\"ahler polarization
is the fact that the canonical bundle $K_{X}$ of a toric
variety may not admit a square root, for instance for ${\mathbb{C}} {\mathbb{P}}^{2n}$. (See the appendix for
a discussion of the existence of $\sqrt{K_X}$ in terms of the fan of $X$.)
In Section \ref{sec:corrQ}, we consider K\"ahler quantization of a compact toric manifold $X$ with
symplectic structure $\omega$ such that
$\frac{[\omega]}{2\pi}-\frac{c_{1}(X)}{2}\in H^{2}(X,{\mathbb{Z}})$ and is nonnegative.
(This integrality condition has also been proposed in \cite{C78}.)
In the case when
$c_{1}(X) $ is even, so that $K_{X}$ admits a square root, one is then reduced
to the usual setting for half-form quantization.
Let $L\rightarrow X$ be an hermitian line bundle with connection of curvature
given by $-i{\omega}+\frac{i}{2}\rho$, where $\rho$ is the Ricci form for the
K\"ahler metric on $X$, so that $[\rho/2\pi]\in c_{1}(X)$. When a $\sqrt
{K_{X}}$ exists, this corresponds to taking the usual prequantum connection
plus one-half the Chern--Levi--Civita connection on $K_{X}$, which gives a connection
on $\sqrt{K_{X}}$.

The condition $\frac{[\omega]}{2\pi}-\frac{c_{1}(X)}{2}\in H^{2}(X,{\mathbb{Z}})$ allows us to choose
the moment polytope (see equations (\ref{t11}) and (\ref{halflambdas}) in Section \ref{corrq}),
$P_X = \mu(X)$,
\be\nonumber
P_X = \left\{x \in \R^n \ : \ \ell_j(x) = \nu_j \cdot x + \lambda_j \geq 0, \ \ j = 1, \dots, r\right\} ,
\ee
with  all $\lambda_j$'s half-integral
\begin{equation}\nonumber
\lambda_j \in \frac 12 + \Z, j=1\dots,r ,
\end{equation}
where $x$ are action coordinates and $\nu_j$ is the primitive inward pointing normal vector to the $j$-th facet.
With this choice, there are no integral points in the boundary of $P_X$ and to all
integral points inside $P_X$ there will correspond K\"ahler polarized states, that is holomorphic sections of $L$.
Unlike Hamilton's case however, the integral points start at lattice distance $1/2$ (rather than $1$), from
every facet (see the figures in Remarks \ref{rmk:CP1} and \ref{rmk:CP2}.)

As in the case without the half-form correction, these polarized states will converge, as the polarizations degenerate to the toric real polarization, to delta distributions supported
on the corresponding non-singular orbits.
The degeneration of the equations for polarized sections (see section \ref{subsec:realQuant})
is also consistent with the degeneration of the polarized states as the
Ricci connection term $ i d \theta^j_v$ in (\ref{limitconnvert}) corresponds to a connection with
curvature supported on the inverse image of the $j$-th facet by the moment map.
From the point of view of the real polarization these singular connections are
responsible for the vacuum energy shifts (which correspond to shifted Bohr--Sommerfeld conditions)
as they prevent
the existence of covariantly constant sections supported on the boundary.
Quantization in the real singular toric polarization is then defined
directly in terms of this limit connection. This provides an approach for defining
half-form corrected quantization in real singular polarizations. By finding the
type of singularities of the half-form corrected limit connection, one finds
corrected equations for the real polarized sections. In the toric case, this direct approach
for the definition of the half-form corrected quantization in the singular toric real polarization
gives the same results as the degeneration of K\"{a}hler polarizations (Theorems \ref{thm:realQ} and
\ref{thmdistrib}).

In \cite{Baier-Florentino-Mourao-Nunes}, the convergence to delta-distributions
was achieved by taking $L^{1}$-normalized sections.
In the present paper, however, the half-form correction ensures the nice
behavior of the $L^{2}$-normalized sections in the limit of degenerating
complex structure. This is in agreement with other examples such as
finite-dimensional vector spaces \cite{Kirwin-Wu06} and abelian varieties
\cite{Baier-Mourao-Nunes}.

One of the primary motivations for including the half-form correction is that
it allows for a canonical pairing between quantizations associated to
different complex structures. This pairing is known as the
Blattner--Konstant--Sternberg (BKS) pairing. The BKS pairing between
quantizations associated to two K\"ahler complex structures is nondegenerate,
and hence (since the K\"{a}hler quantizations of a compact toric manifold are
finite dimensional) induces an isomorphism between them. One does not, though,
expect in general that the BKS pairing provides a \textit{unitary}
identification of quantizations associated to different complex structures. In
several common cases, for example for symplectic vector spaces equipped with
translation invariant K\"ahler structures, and for complex Lie groups equipped with
certain families of K\"ahkler structures (which include the canonical
K\"ahler structure), the BKS pairing is unitary (see \cite{Hall02},
\cite{Kirwin-Wu06},\cite{Florentino-Matias-Mourao-Nunes05}%
,\cite{Florentino-Matias-Mourao-Nunes06}). In a few other cases, the BKS
pairing is known to be not unitary, for example for $T^{\ast}S^{2}$
\cite{Rawnsley79}. In most cases, it is not known whether the BKS pairing is
unitary, and conditions for unitarity do not yet seem to be well understood.
We show that the BKS pairing between half-form corrected quantizations of
compact toric varieties is not unitary in general.

We will also consider another method for comparing quantizations associated to
different complex structures. Namely, one can construct a (finite-rank)
Hilbert bundle over the space of toric complex structures
on $X$ with the fiber at a point being the quantum
Hilbert space associated to that complex structure. When the half-form correction is not included,
the quantum Hilbert bundle is a subbundle of a trivial bundle, and hence carries a canonical
connection obtained by
orthogonal projection of the trivial connection. This connection is called the
\emph{quantum connection.} It was first introduced and studied by Axelrod,
della Pietra and Witten in \cite{Axelrod-DellaPietra-Witten} and, from a
slightly different point of view, by Hitchin in \cite{Hitchin90}. See also
\cite{Andersen-Gammelgaard-Lauridsen} for a treatment which includes
the half-form correction.

For linear complex structures on a symplectic vector space, the quantum
connection turns out to be projectively flat, which means that up to a
constant, one may identify all K\"{a}hler quantizations at once.
To extend the connection to the boundary of the space of complex structures,
and thus study their degenerations
and relate real quantizations to K\"{a}hler quantizations,
one must introduce the half-form correction; parallel transport of the
resulting corrected quantum connection, still in the case of linear complex
structures, was studied by the first author and Wu in \cite{Kirwin-Wu06},
where it was found that parallel transport along geodesics with internal
endpoints is just rescaled Bergman projection, while transport along geodesics
with one or two endpoints on the boundary yields the well-known
Segal--Bargmann and Fourier transforms, respectively. These results were recently
extended by Wu to the case of linear quantization of fermions \cite{Wu10}. In
the usual (bosonic) case, one may then quotient by the action of
$\mathbb{Z}^{2n},$ as done by Baier and the second two authors in
\cite{Baier-Mourao-Nunes}, to study degenerations of complex structures on
abelian varieties at the level of $\vartheta$-functions. In a different
direction, in \cite{Florentino-Matias-Mourao-Nunes05} and
\cite{Florentino-Matias-Mourao-Nunes06}, Florentino, Matias and the second two
authors studied the corrected quantum connection on a one-dimensional family
of complex structures on the complexification of a compact Lie group which
degenerates to the vertical polarization of the cotangent bundle of the
underlying real Lie group; here, again, parallel transport with respect to the
quantum connection yields the (generalized) Segal--Bargmann--Hall transform.
In related work, Lempert and Sz\H{o}ke have recently studied
the bundle of quantizations associated to a family of adapted-type complex
structures on Grauert tubes of compact, real-analytic Riemannian manifolds
\cite{Lempert-Szoke10}, although they use a Chern-type connection rather than the BKS construction considered here.

In Subsection \ref{sec:connection},
we show that the quantum connection on the quantum Hilbert bundle induced by the BKS pairing
is flat, so that the quantizations associated to different torus-invariant complex structures
can be canonically identified.

\section{Preliminaries.}

\subsection{\label{subsec:cplx line bdls}Complex line bundles}

We begin with some facts about complex line bundles. Let $E\rightarrow X$ be a
complex line bundle on a manifold $X$ and let $E_{0}=E\setminus\{zero \,\, section\}$ be its frame bundle.
The isomorphism
\begin{align*}
\label{polar}(|\cdot|,arg)\ :\ {\mathbb{C}}^{*}  &  \cong{\mathbb{R}}^{+}
\times U(1)\\
c  &  \mapsto\left(  |c|,\frac{c}{|c|}\right)  ,
\end{align*}
defines a canonical isomorphism (\cite{Weinstein04}, page 6)
\begin{equation}
\label{Edecomp}E\cong|E|\otimes E^{U(1)},\nonumber
\end{equation}
where the complex line bundles $|E|,E^{U(1)}$ are associated to the principal $\C^*$-bundle $E_{0}$, via
the homomorphisms ${\mathbb{C}}^{*} \ni c\mapsto|c|\in{\mathbb{R}}^{+}$ and
${\mathbb{C}}^{*}\ni c\mapsto arg(c)=\frac{c}{|c|}\in U(1)$ respectively.
Following \cite{Weinstein04} we call the line bundle $E^{U(1)}$ the unitarization of $E$.

This isomorphism is given explicitly by
\begin{equation}
E_{p} \ni l \mapsto\left\{
\begin{array}
[c]{ll}%
|l|\otimes l^{U(1)}, & l\neq0\\
0, & l=0,
\end{array}
\right. \nonumber
\end{equation}
where $p\in X, |l|=[(l,1)]_{|\cdot|}=[(lc^{-1},|c|)]_{|\cdot|} \in
|E|=E_{0}\times_{({\mathbb{C}}^{*},|\cdot|)} {\mathbb{C}}$ and $l^{U(1)}%
=[(l,1)]_{arg}=[(lc^{-1},\frac{c}{|c|})]_{arg} \in E^{u(1)}=E_{0}%
\times_{({\mathbb{C}}^{*},arg)} {\mathbb{C}}, c\in{\mathbb{C}}^{*}$.

For simplicity, we will identify $E$ with $|E|\otimes E^{U(1)}$ and write
$0\neq l=|l|\otimes l^{U(1)}=|l|l^{U(1)}$ and thus, also, $l^{U(1)}=\frac
{l}{|l|}$.

Let $\{g_{\alpha\beta}\}$ be the transition functions for $E$ associated to
local trivializations for some open cover $\{U_\alpha\}$ of $X$.
Then, for the same open cover $\{U_\alpha\}$, the complex line bundle $E^{U(1)}$ has $U(1)$-valued
transition functions $\{g_{\alpha\beta}/\left\vert g_{\alpha\beta}\right\vert
\}$, and the complex line bundle $\left\vert E\right\vert $ has ${\mathbb{R}}^{+}$-valued
transition functions $\{\left\vert g_{\alpha\beta}\right\vert \}$.

This decomposition of $E=|E|\otimes E^{U(1)}$ induces an associated splitting
of connections. Let $\nabla$ be a connection on $E$ with connection form
$\Theta$ associated to a local trivializing section $s$, $\nabla s=\Theta s$.
Since, at the level of Lie algebras, the isomorphism ${\mathbb{C}}^{\ast}%
\cong{\mathbb{R}}^{+}\times U(1)$ gives ${\mathbb{C}}\cong{\mathbb{R}}\oplus
i{\mathbb{R}}$, we have $\nabla=\nabla^{|E|}+\nabla^{E^{U(1)}}$ where
$\Theta^{|E|}=\operatorname{Re}\Theta$ and $\Theta^{E^{U(1)}}=i\Im\Theta$ are
the connection forms for $|E|,E^{U(1)}$ associated to the local trivializing
sections $|s|,s^{U(1)}$ respectively:
\begin{align*}
\nabla^{|E|}|s|  &  =\operatorname{Re}\Theta\,|s|\\
\nabla^{E^{U(1)}}s^{U(1)}  &  =i\Im\Theta\,s^{U(1)}.
\end{align*}

Now let $E$ have an hermitian structure $h$. Then $|E|$ has a global
trivializing section $\mu_{h}$ defined as follows. Let $s$ be any local
trivializing section of $E$ over an open set $U\subset X$. Over $U$, define
\begin{equation}
\mu_{h}=\frac{|s|}{\sqrt{h(s,s)}}. \label{lemma:|E|triv-sec}%
\end{equation}
Notice that, since $h$ is an hermitian structure, $\mu_{h}$ is independent of the choice of
the local trivializing section $s$ and
therefore extends to a global trivializing section of $|E|.$

Let $\Gamma(E)$ denote the space of smooth sections of $E$.
A connection $\nabla$ on $(E,h)\rightarrow X$ is said to be compatible with
the hermitian structure if for any section $s\in\Gamma(E)$, one has
$dh(s,s)=h(\nabla s,s)+h(s,\nabla s)$. Let $||s||^{2}=h(s,s)$. This property
is equivalent to $d||s||=\operatorname{Re}\Theta||s||$ which is, in turn,
equivalent to
\[
\nabla^{|E|}\mu_{h}=0.
\]

\begin{remark}
\label{quotas} The above isomorphism of $|E|$ with the trivial bundle also
defines, since $E=|E|\otimes E^{U(1)}$, an isomorphism of $E$ with $E^{U(1)}$
given by $l\mapsto l\frac{\sqrt{h(l,l)}}{|l|}, l\in E$.
\end{remark}

\begin{remark}
If $X$ is a complex manifold and if $E$ has an holomorphic structure, then
given a global nonzero meromorphic section of $E$, $s$, one has $\mu_{h} =
\frac{|s|}{\sqrt{h(s,s)}}$ away from the divisor of $s$. This expression
extends uniquely to $\mu_{h}$ on the whole of $X$.
\end{remark}

Any line bundle $E^{U(1)}$ has a canonical hermitian structure defined by
$\widetilde{h} (zl^{U(1)},z^{\prime}l^{U(1)})= z \bar z^{\prime}$. This hermitian structure
is independent of the choice of representative $l$
and is therefore well defined.

\begin{remark}
An hermitian line bundle $(E,h)\rightarrow X$ can then be decomposed into
smooth hermitian line bundles $(E,h)=(\left\vert E\right\vert ,\widehat
h)\otimes(E^{U(1)},\widetilde{h})$. The hermitian structure on $|E|$ is
defined by $\widehat h(z|l|,z^{\prime}|l|)= h(l,l) z \bar z^{\prime}$, so
that
\[
h(zl,z^{\prime}l)= z\bar z^{\prime}\, \widehat h(|l|,|l|) \cdot\widetilde
h(l^{{U(1)}},l^{U(1)})= z\bar z^{\prime}\, \widehat h(|l|,|l|),\,\,\,
z,z^{\prime}\in{\mathbb{C}}, l\in E.
\]
\end{remark}

\begin{remark}
Note that under the isomorphism $\left\vert E\right\vert \simeq{X\times
\mathbb{C}}$ defined by $\mu_{h}$, the hermitian form $\widehat h$ becomes the
standard hermitian product on ${\mathbb{C}}$.
\end{remark}

We recall the following standard results (see, for instance, Proposition 4.2.14 in
\cite{Huybrechts}):

\begin{lemma}
If $(E,h)\rightarrow X$ is a complex hermitian vector bundle, then there
exists a compatible connection $\nabla.$ Moreover, $\nabla^{\prime}$ is
another compatible connection if and only if there exists a (global)
real-valued $1$-form $\beta$ such that $\nabla^{\prime}=\nabla+i\beta.$
\end{lemma}

If $E\rightarrow X$ is a holomorphic line bundle, then $\bar{\partial}$ is a
well-defined operator on $\Gamma(E)$. A connection $\nabla$ on $E$ is said
to be compatible with the holomorphic structure if $\nabla^{0,1}=\bar
{\partial}$.

\begin{lemma}
\label{lemma:chern connection} Let $(E,h)\rightarrow X$ be a hermitian
holomorphic line bundle. There exists a unique connection $\nabla,$ called the
Chern connection, which is compatible with both the hermitian structure $h$
and the holomorphic structure of $E$. Moreover, if $\{U_{\alpha},s_{\alpha}\}$
is a holomorphic trivialization of $E$, then $\nabla s_{\alpha}=\left(
\partial\log h(s_{\alpha},s_{\alpha})\right)  \,s_{\alpha}.$
\end{lemma}

Then, in the holomorphic local trivialization $\{U_{\alpha},s_{\alpha}\}$, we
have
\[
F_{\nabla}=d(\partial\log h(s_{\alpha},s_{\alpha}))=-\partial\bar{\partial
}\log h(s_{\alpha},s_{\alpha})
\]
that is, $-\log h(s_{\alpha},s_{\alpha})$ is a local potential for the
curvature $2$-form and, on the open set $U_{\alpha},$
\[
[i\partial\bar{\partial}(-\log h(s_{\alpha},s_{\alpha}))] \in2\pi\cdot
c_{1}(E).
\]

The induced connections on $|E|$ and $E^{U(1)}$ are then given by
\begin{align*}
\nabla^{|E|} |s|  &  = \frac12 d\log h(s,s)\, |s|,\\
\nabla^{E^{U(1)}} s^{U(1)}  &  = \frac12 \left(  \partial\log h(s,s) -
\bar\partial\log h(s,s)\right)  \, s^{U(1)}.
\end{align*}

If $X$ is K\"ahler with integral symplectic form $\omega$ , and
$L$ is an hermitian holomorphic line bundle with the curvature of the Chern
connection given by $-i \omega$, then in a local holomorphic trivialization one
has that $\kappa=-\log h(s,s)$ is a local K\"{a}hler potential.

\subsection{Toric Manifolds}

Let $(X,\omega)$ be a compact smooth symplectic toric manifold with symplectic
form $\omega$, moment map\\ $\mu:X\rightarrow Lie(\mathbb{T}^{n})^{\ast}%
\simeq\mathbb{R}^{n}$ and moment polytope $P_{X}=\mu(X)$ with associated fan
$\Sigma$.
The K\"{a}hler structure of $X$, which connects the
symplectomorphism class of $X$ determined by $P_{X}$ to the biholomorphism
class of $X$ determined by $\Sigma$, is fixed by choosing a so-called
symplectic potential. We will find both descriptions, as well as the relation
between them, essential for our work, and so we describe them briefly here (though we
refer the interested reader to \cite{Guillemin94}, \cite{Abreu03},
\cite{Cox} and \cite{DP09} for details).

\subsubsection{The symplectic structure of $X$}

\label{maldives}

Let $\check{P}_{X}$ denote the interior of the moment polytope $P_{X}$. On
$\check{X}=\mu^{-1}(\check P_X)\cong\check{P}_{X}\times{\mathbb{T}}^{n}$ consider action-angle
coordinates $(x,\theta)$, so that $\mu({x},\theta) = x = \,^{t}(x^{1}%
,\dots,x^{n})$. The symplectic form $\omega$ in this coordinate chart is
simply
\begin{equation}
\left.  \omega\right\vert _{\mu^{-1}(\check{P}_{X})}=\sum_{j=1}^{n}
dx^{j}\wedge d\theta^{j}. \label{eqn:interior symp}%
\end{equation}

The moment polytope $P_{X}$ is a Delzant polytope (see \cite{D88} or page 698 of \cite{DP09})
determined by a set of inequalities $\{\ell
_{j}({x})\geq0\}_{j=1,\dots,r}$, where $r$ is the number of
facets of $P_{X}$ and for each $j=1,\dots,r$,
\[
\ell_{j}(x)=\nu_{j}\cdot{x}+\lambda_{j}
\]
where $\nu_{j}$ is the (inward pointing) primitive integral vector normal to the $j$-th facet of
$P_{X}$, and $\lambda_{j}\in\mathbb{R}.$

We now describe the coordinate chart associated to a vertex $v\in P_{X}$.
Since we assume $X$ is smooth, the polytope is regular; that is, there are $n$
facets adjacent to each vertex, with normal vectors forming a $\mathbb{Z}%
$-basis of $\mathbb{Z}^{n}$. Reorder (if necessary) the inequalities so that
the first $n$ correspond to the facets adjacent to $v$. Then $\ell_{1}%
(v)=\ell_{2}(v)=\dots=\ell_{n}(v)=0.$ Let $A_{v}\in GL_{n}({\mathbb{Z}})$ be
the matrix whose rows are the vectors $\nu_{j},$ and let ${\lambda}%
_{v}=\,^{t}(\lambda_{1},\dots,\lambda_{n})$. Define new (vertex action-angle) coordinates ${x}_{v}$
on $\mathbb{R}^{n}$ and ${\theta}_{v}$ on $\mathbb{T}^{n}$ by
\begin{equation}
{x}_{v}:=A_{v}{x}+{\lambda}_{v}\text{, and } \theta_{v}:=\,^{t}A_{v}%
^{-1}{\theta}. \label{eqn:v-coord change}%
\end{equation}
The image of the polytope $P_X$ under $x_v$ in (\ref{eqn:v-coord change})
is also a Delzant polytope  $P^v_X$,
\be
\label{pxve}
P^v_X = A_v \ P_X + \lambda_v ,
\ee
 with the vertex $v$ mapped to the
origin and the codimension one faces meeting at the origin contained in the
coordinate hyperplanes.
Given a different vertex $v'$, with associated matrix $A_{v'} \in GL_n({\mathbb{Z}})$,
and vector $\lambda_{v'}$, the transition functions
between the corresponding vertex action angle coordinates read
\begin{eqnarray}
\nonumber
x_{v'} &=& A_{v'}A_v^{-1}(x_v-\lambda_v) + \lambda_{v'} \\
\theta_{v'} &=& {}^tA_{v'}^{-1} \ {}^tA_v \ \theta_v  . \label{trf23}
\end{eqnarray}
For a face $F\subset P_{X}$ (that is, a linear boundary component of any
codimension, including the polytope itself, the facets and the edges), denote
by $\check{F}$ the interior of $F$, with the convention that $\check{\{v\}}=\{v\}$
 for vertices.
The vertex chart neighborhood at $v$ is defined to be
the following $\mathbb{T}^n$-invariant open set
\[
U_{v}:=\mu^{-1}\left(\bigcup\limits_{\text{faces }F\text{ of }P_{X}
\text{ adjacent to }v}\check{F}\right)  .
\]
We consider on $U_{v}$ coordinates $\{a_v^j,b_v^j\}_{j=1,\dots,n}$ related to
the vertex action-angle coordinates $\{{x}_{v}^j,\theta_{v}^j\}_{j=1,\dots,n}$ by
$a_v^j+ib_v^j= \sqrt{x_v^j}e^{i\theta_v^j}, j=1,\dots,n$, on $\check X$
(see, for example, Sections 3 and 4 of
\cite{DP09}).
Since $x_v$ takes values in the polytope $P_X^v \subset \R^n$ (\ref{pxve}),
(it is surjective to the polytope minus the faces not containing the origin) and $\theta_v \in \R^n/\Z^n$,
the image of $U_v$ under $a_v+ib_v$ is a bounded neighborhood of the origin
in $\C^n$.
The (non-holomorphic) transition functions between coordinate functions $a_v+ib_v$ and
$a_{v'}+ib_{v'}$ for vertices $v$ and $v'$ can be obtained from
(\ref{trf23}) (see Section 4 of \cite{DP09}),
$$
a^{j'}_{v'}+ib^{j'}_{v'} = \sqrt{\sum_{j=1}^n (A_{v'} \ A_v^{-1})_{j'j} \ \left(|a_v^j+ib_v^j|^2-\lambda_{v}^{j}\right) + \lambda_{v'}^{j'}}  \,\,\,\, \prod_{j=1}^n\left(\frac{a_v^j+ib_v^j}{|a_v^j+ib_v^j|}\right)^{({}^tA_{v'}^{-1} \ {}^t A_v)_{j'j}}.
$$
We will also need the much simpler transition functions for holomorphic vertex
coordinates (which will be introduced below in the section on K\"ahler structures).

We note
that the faces of $P_{X}$ correspond to points in $X$ with nontrivial
stabilizer as follows: suppose $F$ is a face adjacent to $v$ given by
$\{x_{v}^{j_{s}}=0\}_{s=1,\dots,j_{F}}$ (so $F$ is a codimension-$j_{F}$
face). Then the points in $\mu^{-1}(F)$ are fixed by the subtorus
parameterized by the coordinates $\{(\theta_{v}^{j_{1}},\dots,\theta
_{v}^{j_{F}})\}$. Let $V$ be the set of vertices of $P_X$.
We call $\{(\mu^{-1}(\check{P}_{X}),(x,\theta)),(U_{v},(a_v,b_v)):v \in V\}$ the \emph{vertex atlas} of $X$.

The symplectic form in the vertex coordinate chart $U_{v}$ can be computed by
pullback of (\ref{eqn:interior symp}) under the coordinate change
(\ref{eqn:v-coord change}) to be
\[
\left.  \omega\right\vert _{U_{v}}= \sum_{j=1}^n 2da_v^j\wedge db_v^j,
\]
and on $U_v\cap \check X=\check X$,
\[
\left. \omega\right\vert_{\check X}=\sum_{j=1}^{n} dx_{v}^{j} \wedge
d\theta_{v}^{j}.
\]

\subsubsection{The K\"{a}hler structure of $X\label{subsec:Kaehler}$}

In order to describe the toric K\"{a}hler structures on $X$, let us consider torus-invariant complex
structures on the symplectic toric manifold $(X,\omega)$ with moment
polytope $P_{X}$.
Let $g_{P_{X}}\in C^{\infty}(\check{P}_{X})$ be
\begin{equation}
g_{P_{X}}({x})=\frac{1}{2}\sum_{j=1}^{r}\ell_{j}({x})\log\ell_{j}%
({x}).\label{eqn:gP}
\end{equation}

\begin{definition}
Let $C_{P_{X}}^{\infty}(P_{X})$ be the set of smooth functions on $P_X$ such that
$\varphi\in C_{P_{X}}^{\infty}(P_{X})$ if
$\operatorname{Hess}_{x}(g_{P_{X}}+\varphi)$ is positive definite on $\check{P}_{X}$ and
there exists a strictly positive function $\alpha\in C^\infty(P_{X})$ so that
\begin{equation}
\det(\operatorname{Hess}_{x}(g_{P_{X}}+\varphi))=\left(  \alpha({x})
\prod_{j=1}^{r}\ell_{j}({x})\right)^{-1}\label{eqn:greg}
\end{equation}
on $\check P_X$.
\end{definition}
A torus-invariant complex structure on $(X,\omega)$ is determined by a
symplectic potential
\[g=g_{P_{X}}+\varphi,\]
with $\varphi\in C_{P_{X}}^{\infty}(P_{X})$, see Section 4 of
\cite{Guillemin94} and Theorem 2.8 of \cite{Abreu03}.
In the symplectic frame determined by the action-angle coordinates
$({x,\theta})$ on $\check{X}$, the toric complex structure $I$ and the metric
$\gamma=\omega(\cdot,I\cdot)$ tensors associated to the symplectic potential
$g$ are then
\begin{equation}\label{ga}
I=%
\begin{pmatrix}
0 & -G^{-1}\\
G & 0
\end{pmatrix}
\text{ and }\gamma=%
\begin{pmatrix}
G & 0\\
0 & G^{-1}%
\end{pmatrix}
,
\end{equation}
where $G=\operatorname{Hess}_{{x}}g$ is the Hessian of $g$.

Let us now relate these complex structures to the algebro-geometric description of toric manifolds.
By a standard construction, see, for example, Section 5 in \cite{D88} or Definition 6.4.2 in \cite{CdS},
associated to the moment polytope $P_X$ there is an associated complete fan $\Sigma$. This fan defines a compact smooth
toric variety $Y$ diffeomorphic to $X$ (see below) and with canonical complex
structure defined by $\Sigma$.

The complex torus $(\mathbb{C}^{*})^{n}$ acts on $Y$ with a dense open orbit
(biholomorphic to $(\mathbb{C}^{*})^{n}$) which we henceforth refer to as
\emph{the open orbit}. Let $M$ denote the (integer lattice of) characters of
$(\mathbb{C}^{*})^{n},$ so that after a choice of basis, $M\simeq
\mathbb{Z}^{n}.$\footnote{We will henceforth identify $M\cong \Z^n$ and $M\otimes \R \cong \R^n$.}
The characters of $(\mathbb{C}^{*})^{n}$ extend to meromorphic functions on $Y$ with
torus-invariant divisors.

The toric variety $Y$ has an atlas of holomorphic coordinates $\{(V_v,\tilde w_v)\}_{v\in V}$,
$\tilde w_v=(\tilde w_v^1,\dots,\tilde w_v^n)$,
where for each pair of vertices $v,v'$, over $V_v\cap V_{v'}$ the glueing conditions are given by
\begin{equation}
\tilde w_{v^{\prime}}=\tilde w_{v}^{A_{v}A_{v^{\prime}}^{-1}}, \label{eqn:wvvprime}
\end{equation}
with $A_{v}A_{v^{\prime}}^{-1}$  interpreted as a row of multiindices; i.e.
$\tilde w_{v^{\prime}}^{j}=\prod_{l=1}^{n}(\tilde w_{v}^{l})^{(A_{v}A_{v^{\prime}}^{-1}%
)_{j}^{l}}$. (See, for example, Section 5 of \cite{DP09}.)

Denote the open orbit in $Y$ by $V_{0}.$ The symplectic potential $g$ fixes a
biholomorphism $(\check{X},I) \cong\check{P_{X}}\times\mathbb{T}^{n} \cong
V_{0}\cong(\mathbb{C}^{*})^{n}$ given by
\ba
\label{holst}
\nonumber \check{P}_{X}\times\mathbb{T}^{n} &\longrightarrow& V_{0}\simeq
(\mathbb{C}^{*})^{n} \\
(x, \theta) &\mapsto &\tilde w =e^{y+i{\theta}%
}=(e^{y^{1}+i\theta^{1}},\dots,e^{y^{n} +i\theta^{n}}) ,
\ea
where $y^{j}=\partial g/\partial x^{j}$. Note that this map is not a symplectomorphism
with respect to the standard symplectic structure on $(\mathbb{C}^{*})^{n}$.

The map ${x}\mapsto{y}=\partial
g/\partial{x}$ \ is a bijective Legendre transform. The inverse map is given
by ${x}=\partial h/\partial{y}$, where $h$ is a K\"{a}hler potential given in
terms of $g$ by
\[
h:={x}\cdot{y}-g\text{.}%
\]
This biholomorphism extends uniquely to a biholomorphism $$\psi_g:X \to Y$$ as follows.
The complex structure associated to $g$ defines holomorphic coordinates in the
vertex coordinate charts via the coordinate change (\ref{eqn:v-coord change})
to yield%
\ba
\nonumber
U_v  &\longrightarrow& V_{v} \\
\label{bihol}
({x}_{v},{\theta}_{v}) & \mapsto &  \tilde w_{v}=e^{{y}_{v}+i{\theta}_{v}} \ ,
\ea
where $y^{j}_{v}:=\partial g/\partial x_{v}^{j}=\sum_{k=1}^{n} \,\left(
A_{v}^{-1}\right)  _{j}^{k}\partial g/\partial x^{k}$.
Using the observation that $y_{v}+i\theta_{v}=\,^{t}A_{v} ^{-1}(y+i{\theta)}$,
one may verify that (\ref{eqn:wvvprime}) is indeed satisfied. Similarly, on $V_0$ we have
\begin{equation}
\tilde w= \tilde w_{v}^{A_{v}}, \label{eqn:ww_v}
\end{equation}
which will be useful below.

We define the $I$-dependent {\it holomorphic vertex atlas} $\{(U_v,w_v)\}_{v\in V}$ on $X$ to be the pullback by $\psi_g$
of the holomorphic atlas $\{(V_v,\tilde w_v)\}_{v\in V}$ on $Y$. (We will also denote the pullback of the chart
$(V_0,\tilde w)$ on $Y$ to $X$ by $(U_0,w)$.) The $I$-dependent transition functions for the holomorphic coordinate charts
$(U_v,w_v), v\in V, (U_0,w)$ on $X$ are therefore the pullbacks by $\psi_g$ of the corresponding transition
functions on $Y$ in (\ref{eqn:wvvprime}) and (\ref{eqn:ww_v}).

Henceforth, we will assume that $X$ is equipped with a K\"ahler structure
determined by $\omega$ and by a symplectic potential $g=g_{P_X}+\varphi,\ \varphi\in C^\infty_{P_X}(P_X).$

\subsubsection{Line bundles and sections\label{subsec:line-bundles}}

Since compact smooth toric varieties
are simply connected, the Picard group of equivalence classes of holomorphic
line bundles is isomorphic to $H^2(X, \Z)$, with isomorphism established
by the first Chern class. In other words, fixing the first Chern class of a line bundle on the
the complex toric manifold $X$ fixes the bundle up to isomorphism. (See the Corollary on p.64,
on Section 3.4 of \cite{F}.)

The linear equivalence classes of the torus-invariant divisors of $X$ generate the Picard group of $X$, and
there is a one-to-one correspondence between irreducible torus-invariant
divisors and $1$-cones in $\Sigma$  (see, Part I, Chapter 4 of \cite{Cox}) .
Denote the set of $1$-cones in $\Sigma$ by
$\Sigma^{(1)}. $ The $j$-th $1$-cone in $\Sigma^{(1)}$ is generated by the primitive
integral vector $\nu_{j}$ normal to the $j$-th facet of $P_X$. Then, the associated
irreducible divisor $D_j=\mu^{-1}(\{x\in P_X: \ell_j(x)=\nu_j\cdot x +\lambda_j=0\})$ is the inverse image
under the moment map $\mu$ of that facet of $P_{X}$.
The Picard group is then generated by the linear equivalence classes of
irreducible divisors $D_{1},\dots,D_{r}$.
Consider a divisor $D^L=\lambda_{1}^{L} D_{1}+\dots+\lambda_{r}^{L}
D_{r}$, for $\lambda_{1}^{L},\dots,\lambda_{r}^{L} \in{\mathbb{Z}}$, defining
a holomorphic line bundle $L=\mathcal{O}(D^L)$ and a
(unique up to constant) meromorphic section of $L$ with divisor $D^L$, $\sigma_{D^L}$ .

From \cite{Cox}, the divisor of the (rational) function defined on the
open orbit by $w^{m}, m\in \Z^n$, can be computed to be
\begin{equation}
\operatorname{div}(w^{m})=\sum_{j=1}^{r} \left\langle \nu_{j},m\right\rangle
D_{j}. \label{eqn:Coxformula}
\end{equation}

Then, we have
\begin{eqnarray}\label{hzero}\nonumber
H^0(X,L)&=& {\rm span}_\C \left\{w^m \sigma_{D^L}: m\in\Z^n,\,\operatorname{div}(w^m_0 \sigma_{D^L})
\geq 0 \right\}
= \\
&=& {\rm span}_\C
\left\{w^m \sigma_{D^L}: m\in\Z^n, \langle m,\nu_i\rangle+\lambda_i^L\geq 0, i=1,\dots,r \right\}.
\end{eqnarray}
Therefore, there is a natural bijection between a basis of $H^0(X,L)$ whose elements are
weight vectors for the action of the torus and
the integral
points of the Delzant polytope with integral vertices \footnote{Note that we have different sign
conventions for $\lambda_{F}$ than those used in
\cite{Baier-Florentino-Mourao-Nunes}. We have chosen rather to follow the
convention in \cite{Cox}, as they seem to make certain equations more natural
(for example, shifting $\lambda_{F}\mapsto\lambda_{F}+1$ has the effect of
shifting the facet $F$ one unit along the \emph{outward} pointing normal to
$F$).}
\begin{equation}
P_{L}:=\{x\in \R^n:\left\langle x,\nu_{j}\right\rangle +\lambda
_{j}^{L} \geq0,~j=1,\dots,r\}\subset \R^n. \label{eqn:p_Ldef}%
\end{equation}

For simplicity, let us assume that $L$ is ample so that there is a canonical bijection between
the vertices of $P_L$ and the vertices of $P_X$, defined by the equality of the set of
normals of the facets meeting at those vertices.
(In fact, if $L$ is ample there is a bijection between the faces of $P_X$ and of $P_L$. See Section
3.2.1 of \cite{CK}.)
Let us denote by the same symbol $v$ a vertex
of $P_L$ and the corresponding vertex of $P_X$.
The holomorphic section corresponding to the vertex $v$ of $P_L$ will provide
a local trivializing section on the open set $U_v$, so that one obtains a global
system of local holomorphic trivializations for $L$.
For such vertex $v$, we can order the inequalities $\{\ell_{j}%
^{L}(m):=\left\langle m,\nu_{j}\right\rangle +\lambda_{j}^{L} \geq 0, j=1\dots r\} $ so
that $\nu_{1},\dots,\nu_{n}$ are the normals to the facets of $P_{L}$ meeting
at $v$; this is the same ordering that we used in the definition of the vertex
coordinates on $X$. Using this ordering, we set $\lambda_{v}^{L}%
=\,^{t}(\lambda_{v,1}^{L},\dots,\lambda_{v,n}^{L}):=(\lambda_{1}^{L}%
,\dots,\lambda_{n}^{L}).$

The holomorphic section corresponding to a vertex $v$ of $P_L$
is given by
\begin{equation}\label{umve}
\mathbf{1}_{v}:=w_v^{-\lambda_v^L} \sigma_{D^L}.
\end{equation}
Using (\ref{eqn:ww_v}) and (\ref{eqn:Coxformula}), one obtains that the divisor
of the meromorphic function $w_{v}^{\lambda_{v}^{L}}$ on $U_{v}$ is
\begin{equation*}
\operatorname{div}_{U_{v}}(w_{v}^{\lambda_{v}^{L}})=\left(\lambda_{v,1}^{L}
D_{1}+\cdots+\lambda_{v,n}^{L} D_{n}\right) \cap U_v, \label{eqn:divwL}
\end{equation*}
and therefore ${\rm div}_{U_v}\, (\mathbf{1}_{v})=0$, so that $\mathbf{1}_{v}$ is a trivializing holomorphic
section of $L$ on $U_v = X\setminus \{\cup_{j=n+1}^r D_j\}$. We remark that these sections
are determined up to a constant by their divisors and they are therefore defined for every line
bundle in the isomorphism class of $L$.

For notational convenience, let $\mathbf{1}_0=\sigma_{D^L}$.
Using (\ref{umve}), we
may compute the transition functions for $L$ relative to the holomorphic vertex atlas obtaining, $g_{v^{\prime}v}
^{L}:=\mathbf{1}_v/\mathbf{1}_{v'}=w_{v'}^{\lambda_{v'}^{L}}/w_{v}^{\lambda_{v}^{L}}$ and
$g_{0v}^{L}:=\mathbf{1}_v/\mathbf{1}_0=w_{v}^{-\lambda_{v}^{L}}.$ Combined with (\ref{eqn:wvvprime}), the
transition functions for $\mathcal{O}(\lambda_{1}^L D_{1}+\cdots+\lambda_{r}^L
D_{r})$ become
\begin{align}
g_{v0}^{L}(w)  &  =w^{A_{v}^{-1}\lambda_{v}^{L}} \,\,\, \text{ and}
\label{eqn:Ltransfns}\\
g_{v^{\prime}v}^{L}(w_{v^{\prime}})  &  =w_{v^{\prime}}^{\lambda_{v^{\prime}%
}^{L}-A_{v^{\prime}}A_{v}^{-1}\lambda_{v}^{L}}.\nonumber
\end{align}
\begin{remark}
Note that the transition functions of $L$ depend  on the variation of  complex structure on $X$ through
the symplectic potential $g$, since
$w = e^{\frac{\partial g}{\partial x}+i\theta}$ and
$w_{v} = e^{\frac{\partial g}{\partial x_{v}}+i\theta_{v}}$ (see (\ref{holst}) and  (\ref{bihol})). We will
consider one-parameter families of symplectic potentials,
$g_s = g_{P_X} + \varphi + s \psi, \, \,  \varphi, \psi \in C^\infty_{P_X}(P_X), \, s  \in \R^+$.
The transition functions and therefore $L$ depend smoothly on $s$.
\end{remark}

These relations define a holomorphic line bundle on $X$ for any
integral values of $\lambda_i^L$, even if this line bundle is not ample. In this case,
sections of the sheaf of holomorphic sections over $U_v$ are defined as in (\ref{hzero})
with $X$ replaced by $U_v$ and the divisors $D$ replaced by $D\cap U_v$.

Using the transition
functions (\ref{eqn:Ltransfns}) for $L$, we can give it a concrete realization as
the following equivariant line bundle
\begin{equation}
\label{glue}
L=\left(\bigsqcup\nolimits_{v \in V}U_{v}\times\mathbb{C}\right)  /\sim ,
\end{equation}
where $(w,z)\sim(w^{\prime},z^{\prime})$ if $w=w^{\prime}\in U_{v}\cap
U_{v^{\prime}}$ and $z=g_{vv^{\prime}}^{L}(w)z^{\prime}$.
We will assume that $L={\mathcal O}(D)$, $D=\sum_{j=1}^r\lambda^L_j D_j$, is the line bundle
defined by (\ref{glue}).
In each open set in
the holomorphic vertex atlas, the trivializing sections $\mathbf{1}_{v}$ (or
$\mathbf{1}_{0}$ on the open orbit) defined above are given by
\[
\mathbf{1}_{v}(w):=[(w,1)] \ , \ w \in U_{v}\text{ (or } w \in U_{0} \text{, for the open orbit)}.
\]

For $\sigma\in\Gamma(L)$ denote by $\sigma_v,\sigma_0$ its components on the local frames
given, respectively, by $\mathbf{1}_v,\mathbf{1}_0$.
For an integral point $m\in P_{L}\cap \Z^n$, we denote
by $\sigma^{m}$ the holomorphic section with $\sigma^m_0=w^m$. Using the transition
functions, we obtain expressions for $\sigma^{m}$ in the holomorphic vertex charts:%
\begin{equation}
\sigma_{v}^{m}(w_{v})=w_{v}^{\ell_{v}(m)}\mathbf{1}_{v}, \label{eqn:sigmam_v}%
\end{equation}
where $\ell_{v}({x}):=A_{v}{x}+\lambda_{v}^{L}.$

\bigskip
We have
\begin{lemma}
\label{lemma:I-indep}
Let $L$ be the equivariant holomorphic line bundle  defined by (\ref{eqn:Ltransfns})
and (\ref{glue}).
The unitarization $L^{U(1)}$ associated with $L$ defined over a compact toric variety $X$ has complex
structure independent
transition functions on the vertex atlas.
\end{lemma}

\begin{proof}
Recall from Section \ref{subsec:cplx line bdls} that the unitarization of $L$  is
the line bundle $L^{U(1)}$ with local trivializing sections
$\mathbf{1}_{v}^{U(1)}$ on the vertex charts and $\mathbf{1}_{0}^{U(1)}$ on the open
orbit. The corresponding transition functions are
\begin{equation}
\tilde{g}_{v^{\prime}v}^{L}:=\mathbf{1}_v^{U(1)}/\mathbf{1}_{v'}^{U(1)}=e^{i(\lambda_{v^{\prime}}%
^{L}-A_{v^{\prime}}A_{v}^{-1}\lambda_{v}^{L})\cdot\mathbf{\theta}_{v^{\prime}%
}}\,\,\text{ and }\,\,\tilde{g}_{0v}^{L}:=\mathbf{1}_v^{U(1)}/\mathbf{1}_{0}^{U(1)}=e^{-i(A_{v}^{-1}\lambda_{v}^{L}%
)\cdot\mathbf{\theta}}. \label{eqn:LU(1)transfns}
\end{equation}
We see that, unlike those for $L$ itself,  these transition functions are complex
structure independent.
\end{proof}

\bigskip

Recall also from Remark \ref{quotas} that a hermitian structure on $L$ defines
an isomorphism between $L$ and $L^{U(1)}$.

\subsubsection{The canonical bundle}

\label{canbun}

The complex structure $I$ on $(X,\omega)$ (see (\ref{ga})) defines the canonical holomorphic
line bundle $K_{I}:=\bigwedge^{n}(T^{\ast})^{1,0}$, whose sections are
$(n,0)$-forms. Consider, in the open orbit $U_{0},$ the $I$-holomorphic $(n,0)$-form
$dZ=dz^1\wedge \cdots \wedge dz^n = dW/w^{\mathbf{1}}$, where
\begin{equation}\label{partialz}
z=\,^{t}(z^{1},\dots,z^{n})=\partial g/\partial{x}+i{\theta}
\end{equation}
and $dW:=dw^{1}\wedge\cdots\wedge dw^{n}$, so that $dZ$ and $dW$ are trivializing sections of
${K_{I}}_{|_{U_{0}}}$. (Here, $\mathbf{1}=(1,\dots,1)$ so that
$w^{\mathbf{1}}=w^1\cdots w^n$.) Then

\begin{lemma}\label{lemmacox}
\cite[Sec. 8.2]{Cox} The $(n,0)$-form given on the open orbit by $dZ$ extends to a  meromorphic section of
$K_{I}$\ with divisor ${\rm div}\,(dZ)=-D_{1}-\cdots-D_{r}$. On the holomorphic vertex chart $U_{v}$
this section is proportional to $dW_{v}/w_{v}^{\mathbf{1}}=dZ_{v}$.
\end{lemma}

Since the $w_v^j$'s are holomorphic coordinates on the chart $U_v$, it follows that a system of local
holomorphic trivializations for $K_I$ is given by $\{(U_v,dW_v)\}$.
Relative to this system of trivializations of ${K_{I}}$, the transition functions
are computed to be
\begin{equation}
  g_{v^{\prime}v}^{{K_{I}}}=w_{v^{\prime}}^{-\mathbf{1}+A_{v^{\prime}}A_{v}^{-1}\mathbf{1}},
\label{eqn:rootK_transfnsb}
\end{equation}
so that, as expected from the form of ${\rm div}\,(dZ)$, $K_I$ is isomorphic to a line bundle
of the form of (\ref{glue}).

From Section \ref{subsec:cplx line bdls}, we have that the equivariant
hermitian holomorphic line bundle $K_{I}$ admits a
decomposition $K_{I}^{U(1)}\otimes\left\vert K_{I}\right\vert$.
The unitarization $K^{U(1)}$ is trivialized by
\be
\label{ku1}
\{(U_{v},\frac{dW_v}{|dW_v|})\}_{v\in V}
\ee
 with corresponding
transition functions
\begin{equation}\label{tildeK}
\tilde{g}_{v^{\prime}v}^{{K_{I}}}=e^{i(-\mathbf{1}+A_{v}
A_{v^{\prime}}^{-1}\mathbf{1})\cdot\mathbf{\theta}_{v^{\prime}}}.
\end{equation}
We see that, in accordance with Lemma
\ref{lemma:I-indep}, $K_{I}^{U(1)}$ has $I$-independent transition functions on the
vertex atlas. On the other hand
notice that the  line bundles $ K_{I}^{U(1)}$   depend on $I$
 because  their fibers change with $I$ (see (\ref{ku1}). In order to facilitate the study of
the dependence
of polarized sections on the complex structure $I$, it will be
convenient to consider a  line bundle with the same transition functions as
 $ K_{I}^{U(1)}$  but defined as in (\ref{glue}), so that this line bundle is $I$-independent
(but has an $I$-dependent isomorphism to $ K_{I}^{U(1)}$).
\begin{definition}
\label{tildeku1}
Denote by $\widetilde K^{U(1)}$ the ($I$-independent) equivariant line bundle defined as in (\ref{glue})
with $U(1)$-valued transition functions given by (\ref{tildeK}).
\end{definition}

In the remainder of this section, we continue to consider a fixed toric complex structure $I$, obtained from
a symplectic potential $g$, and will drop the
subscript $I$ for simplicity.
$K$ has a canonical hermitian structure given by comparison
with the Liouville volume form, that is, for an $(n,0)$-form $\eta$,
\[
\left\Vert \eta\right\Vert _{K}^{2}:=\frac{\eta\wedge\bar{\eta}}%
{(2i)^{n}(-1)^{n(n+1)/2}\omega^{n}/n!}.
\]
Let $\nabla^K$ denote the Chern connection corresponding to this hermitian structure. In the above
trivialization, we can compute the connection $1$-form of the Chern
connection (using Lemma \ref{lemma:chern connection}) to be
\[
\partial\log\left\Vert dZ\right\Vert _{K}^{2}.
\]

\begin{lemma}
\label{lemma:normdz}$\left\Vert dZ\right\Vert _{K}^{2}=\det G,$ where
$G=\operatorname{Hess}_{x}g$. Hence, the Chern connection $1$-form in the open
orbit $U_{0}$ is $\Theta_0=\partial\log\det G.$
\end{lemma}

\begin{proof}
Since $z^{j}=\partial g/\partial x^{j}+i\theta^{j},$ we see that
$dz=Gdx+id{\theta}$ and similarly that $d\bar{z}=Gd{x}-id\theta.$ We can
express these in the matrix equation
\[%
\begin{pmatrix}
dz\\
d\bar{z}%
\end{pmatrix}
=%
\begin{pmatrix}
G & i\mathbf{1}\\
G & -i\mathbf{1}%
\end{pmatrix}%
\begin{pmatrix}
dx\\
d\theta
\end{pmatrix}
\]
whence
\[
\left\Vert dZ\right\Vert _{K}^{2}=\frac{dZ\wedge d\bar{Z}}{(-2i)^{n}%
(dx^{1}\wedge\cdots\wedge dx^{n}\wedge d\theta^{1}\wedge\cdots\wedge
d\theta^{n})}=\frac{1}{(-2i)^{n}}\det%
\begin{pmatrix}
G & i\mathbf{1}\\
G & -i\mathbf{1}%
\end{pmatrix}
=\det G.
\]

\end{proof}

Similarly we have
\[
\left\Vert dZ_{v}\right\Vert _{K}^{2}=\det G_{v},
\]
where $G_{v}=\operatorname{Hess}_{x_{v}}g$.

The curvature of the Chern connection is easily computed, giving
\begin{equation}
F_{\nabla^{K}}=\bar{\partial}\partial\log\det G. \label{eqn:curvK}%
\end{equation}
Let $Ric_{\gamma}$ denote the Ricci curvature tensor of the metric
$\gamma=\omega(\cdot,I\cdot)$, and let $\rho=Ric(I\cdot,\cdot)$ be the
corresponding Ricci form. Then by \cite[Prop. 11.4]{Moroianu} we have
\begin{equation}
F_{\nabla^{K}}=i\rho. \label{eqn:curvKrho}%
\end{equation}
Note that this implies $c_{1}(K)=-c_{1} (X)=-[\rho/2\pi]$.

\begin{lemma}
\label{lemma:canonicalconnections}
The section $dW_{v}=w_{v}^{\mathbf{1}}dZ_{v}$ is a nowhere vanishing
holomorphic section of $K$ on the holomorphic vertex chart $U_{v}$ and, in the induced
trivializations, the Chern connection $1$-forms on $K^{U(1)}$ and $\left\vert
K\right\vert $ are
\begin{align}
{\Theta}_{v}^{K^{U(1)}}  &  =i\sum_{k=1}^{n}d\theta_{v}^{k}+\frac{i}
{2}\left(  \frac{\partial}{\partial{x}_{v}}\log\det G_v\right)  \cdot G_{v}%
^{-1}d{\theta}_{v},\text{ and}\label{eqn:canonical1form}\\
\Theta_{v}^{|K|}  &  =\frac{1}{2}\left(  \frac{\partial}{\partial{x}_{v}%
}\log\det G_v\right)  \cdot d{x}_{v}+G_{v}\,d{x}_{v}.\nonumber
\end{align}
\end{lemma}

\begin{remark}
The following expressions for the connection $1$-forms for the induced
connections $\nabla^{K^{U(1)}}$ and $\nabla^{\left\vert K\right\vert }$ over the open orbit
will also be useful below:
\begin{align}
{\Theta}_0^{K^{U(1)}}  &  =i\Im\Theta=\frac{i}{2}\left(  \frac{\partial}
{\partial{x}}\log\det G\right)  \cdot G^{-1}\,d\,{\theta},\text{
and}\label{eqn:openorbit_1-forms}\\
\Theta^{|K|}_0  &  =\operatorname{Re}\Theta=\frac{1}{2}\left(
\frac{\partial}{\partial{x}}\log\det G\right)  \cdot\,d{x} .\nonumber
\end{align}
\end{remark}

\begin{proof}
If $f$ is real valued, then
\begin{equation}
\Im\partial f=\frac{1}{2}\left(\frac{\partial f}{\partial x}\cdot G^{-1}%
d\theta-\frac{\partial f}{\partial\theta}\cdot G\,dx\right) \label{eqn:Imdbar}%
\end{equation}
and
\begin{equation}
\operatorname{Re}\partial f=\frac{1}{2}\left(  \frac{\partial f}{\partial
x}\cdot dx+\frac{\partial f}{\partial\theta}\cdot d\theta\right)  ,
\label{eqn:Redbar}
\end{equation}
with similar formulas in the vertex charts (with $x^{j}$ and $\theta^{j}$
replaced by $x_{v}^{j}$ and $\theta_{v}^{j}$). With $f=\log\det G$ (observing
that $f=f({x})$) we obtain, from $\Theta_0=\partial\log\det G$, the open orbit
$1$-forms (\ref{eqn:openorbit_1-forms}).

Next, by \cite{Cox}, the \textquotedblleft canonical\textquotedblright%
\ section $dZ=dW/w^{\mathbf{1}}$ (which has the same representation
$dW_{v}/w_{v}^{\mathbf{1}}$ in the  holomorphic vertex charts) has simple poles along each
torus-invariant divisor. To obtain a trivializing section on the holomorphic vertex chart
$U_{v}$, we multiply by a factor with simple zeroes along the divisors
adjacent to $v$ and thus arrive at the desired combination $w_{v}^{\mathbf{1}%
}dZ_{v}$.

To obtain the connection $1$-forms in the chart $U_{v}$ with respect to
$w_{v}^{\mathbf{1}}dZ_{v}$, we first recall the norm
\begin{equation}
\left\Vert w_{v}^{\mathbf{1}}dZ_{v}\right\Vert ^{2}=\left\vert w_{v}%
^{\mathbf{1}}\right\vert ^{2}\det G_{v}. \label{eqn:normwdz}%
\end{equation}
From this, noting that $\det G_{v}$ is a constant multiple of $\det G$, we see
that the Chern connection $1$-form on $U_{v}$ is
\[
\Theta_{v}=\partial\log\left(  \left\vert w_{v}^{\mathbf{1}}\right\vert
^{2}\det G_{v}\right)  =\partial\log\det G_{v}+\sum_{j}\partial\log w_{v}%
^{j}=\partial\log\det G_v +\sum_{j}dz_{v}^{j}.
\]
Since $dz_{v}=G_{v}dx_{v}+id\theta_{v}$, using (\ref{eqn:Imdbar}) and that $G$
depends only on ${x}$, we obtain the connection $1$-form
\[
{\Theta}_{v}^{K^{U(1)}}=i\Im\left(  \partial\log\det G_v +\sum_{j}dz_{v}%
^{j}\right)  =\frac{i}{2}\left(  \frac{\partial}{\partial x}\log\det G_v\right)
\cdot G^{-1}_v\,d\theta_v+i\sum_{j=1}^{n}d\theta_{v}^{j}%
\]
as desired.

Similarly, using (\ref{eqn:Redbar}) we obtain the connection $1$-form%
\[
\Theta_{v}^{|K|}=\operatorname{Re}\left(  \partial\log\det G_v+\sum
_{j}dz_{v}^{j}\right)  =\frac{1}{2}\left(  \frac{\partial}{\partial x_v}\log\det
G_v\right)  \cdot\,dx_v+G_{v}\,dx_{v}%
\]
as desired.
\end{proof}

\bigskip

Note that although $K_I^{U(1)}$ has complex structure independent transition functions,
its Chern connection depends on $I$.

\section{Half-form corrected K\"ahler quantization\label{sec:corrQ}}

\subsection{Motivation}
\label{motiv}

Suppose the square root
$\sqrt{K_{I}}$ of the corresponding canonical bundle exists so that, in particular,
$c_{1}(X)/2$ is integral. Assume, moreover, that $[\omega/2\pi]\in
H^{2}(X,{\mathbb{Z}})$ and let $\ell\rightarrow X$ be a (smooth) hermitian
line bundle with compatible connection with curvature given by $-i\omega$,
that is, $\ell$ is a prequantum line bundle. More specifically, let $\ell$ be an equivariant line bundle
defined as in (\ref{glue}), with $U(1)$-valued
transition functions on the  vertex atlas (\ref{eqn:LU(1)transfns})
\begin{equation}
\tilde{g}_{vv^{\prime}}^{\ell}=e^{i\,(A_{v^{\prime}}A_{v}^{-1}\lambda
_{v}-\lambda_{v^{\prime}}) \cdot{\theta}_{v^{\prime}}}\text{, and }\tilde
{g}_{v0}^{\ell}=e^{i\,\lambda_{v}\cdot{\theta}_{v}},
\label{eqn:uncorrected_trans-fns}
\end{equation}
with $\lambda_v \in \Z$ suitably chosen to have $c_1(\ell)=[\frac{\omega}{2\pi}]$.
Equip $\ell$ with the $U(1)$-connection $\nabla^{\ell}$ given by the connection forms
\begin{eqnarray}\nonumber
\Theta_{0}^{\ell}&=&\frac{\nabla^\ell {\mathbf 1}_0^{U(1)}}{{\mathbf 1}_0^{U(1)}}=i{x}\cdot d{\theta}\,\,
{\rm on}\,\, \check X, \\ \label{eqn:uncorrected_conn}
\Theta_{v}^{\ell}&=&\frac{\nabla^\ell {\mathbf 1}_v^{U(1)}}{{\mathbf 1}_v^{U(1)}}=i{x}
_{v}\cdot d{\theta}_{v}\,\, {\rm on}\,\,U_v, v\in V.
\end{eqnarray}
One may easily check that $\{\Theta_{0}^{\ell}, \Theta_{v}^{\ell}: v\in V\}$ does indeed define a
$U(1)$-connection on $\ell;$ see the comment following
(\ref{eqn:Lconndef_open}).

Since the square of a (local) section $\eta\in\Gamma(\sqrt{K_{I}})$ can be
identified with a (local) section of $K_{I}$, the line bundle $\sqrt{K_{I}}$
inherits a hermitian structure from that of $K_{I}$ given by \cite{Woodhouse}
\begin{equation}
\left\Vert \eta\right\Vert _{\sqrt{K_{I}}}^{2}=\sqrt{\frac{\eta^{2}\wedge
\bar{\eta}^{2}}{(2i)^{n}(-1)^{{n(n+1)}/{2}}\omega^{n}/n!}}.
\label{eqn:rootKherm}%
\end{equation}

This defines a Chern connection $\nabla^{\sqrt{K_{I}}}$on $\sqrt{K_{I}}$ as in
Lemma \ref{lemma:chern connection}. The curvature of
$\nabla^{\sqrt{K_{I}}}$ is then $F_{\nabla^{\sqrt{K_{I}}}}=\frac{i}{2}\rho_{I}$,
where $\rho_{I}$ is the Ricci form on $X$.

The quantum Hilbert space for the half-form corrected K\"{a}hler quantization
of $X$ is defined to be
\[
{\mathcal{H}}_{I}^{Q}:=\left\{  s\in\Gamma(\ell\otimes{\sqrt{K_{I}}}):\left(
\nabla_{\overline{\mathcal{P}}_{I}}^{\ell}\otimes1+1\otimes\nabla
_{\overline{\mathcal{P}}_{I}}^{\sqrt{K_{I}}}\right)  s=0\right\}  ,
\]
where $\mathcal{P}_{I}$ is the holomorphic polarization of $X$ determined by
$I$.

Recall from (\ref{lemma:|E|triv-sec}) that the bundle $|\sqrt{K_{I}}|$ has a trivializing covariantly
constant section
\begin{equation}
\mu_{I}=\frac{{|dZ|}^{\frac{1}{2}}}{\left\Vert dZ\right\Vert _{{K_{I}}}%
^{\frac{1}{2}}}. \label{mumu}
\end{equation}

This defines an isomorphism
\[
{\mathcal{H}}_{I}^{Q}\cong\mathcal{B}_{I}^{Q}\otimes\mu_{I}%
\]
where
\begin{equation}\label{categorical}
\mathcal{B}_{I}^{Q}:=\left\{  s\in\Gamma\left(  \ell\otimes \widetilde{\sqrt{K}}^{U(1)}
\right)  :\left(  \nabla_{\overline{\mathcal{P}}_{I}}^{\ell}%
\otimes1+1\otimes\nabla_{\overline{\mathcal{P}}_{I}}^{\widetilde{\sqrt{K}}^{U(1)}} \right)
s=0\right\}  ,
\end{equation}
Note that, from Lemma \ref{lemma:I-indep}, the unitarization $\ell
\otimes \widetilde{\sqrt{K}}^{U(1)}$ is a smooth complex line bundle independent of $I$. In this way,
using the $I$-dependent isomorphisms above, we can describe the quantum
Hilbert spaces ${\mathcal{H}}_{I}^{Q}$ through the Hilbert spaces
$\mathcal{B}_{I}^{Q}$ which are subspaces of a fixed linear space
$\Gamma\left(  \ell\otimes\widetilde{\sqrt{K}}^{U(1)}\right)  $.

We will now use this representation to motivate the definition of the
half-form corrected quantum Hilbert space in the more general situation when
the canonical bundle of $X$ may not admit a square root.

\subsection{Corrected quantization}
\label{corrq}

Let $(X,\omega, I)$ be a compact smooth toric K\"ahler manifold with toric complex
structure $I$ and
such
that $\left[  \frac{\omega}{2\pi}\right]  -\frac{1}{2}c_{1}(X)$ is an ample
integral cohomology class. Let   the moment polytope be
\begin{equation}\label{t11}
P_X = \left\{x \in \R^n \ : \ \ell_j(x) = \nu_j \cdot x + \lambda_j \geq 0, \ \ j = 1, \dots, r\right\} ,
\end{equation}
where we use the freedom of translating the moment polytope to choose
the  $\{\lambda_j\}_{j=1,\dots,r}$ to be half-integral and defined as follows.
Consider an equivariant complex line bundle $L\cong {\mathcal O}(\lambda^L_1 D_1+\cdots+\lambda^L_rD_r)$ as in (\ref{glue}) and
with $U(1)$-valued transition functions (\ref{eqn:LU(1)transfns}) defined by $\{\lambda^L_j\}_{j=1,\dots,r}$, such that
$c_{1}(L) = \left[\frac{\omega}{2\pi}\right]  -\frac{1}{2}c_{1}(X)$. As in (\ref{eqn:p_Ldef}), the $\{\lambda^{L}%
_{j}\}_{j=1,\cdots, r}$ define a polytope with integral vertices, $P_{L}$.
The half-integral
$\{\lambda_j\}_{j=1,\dots,r}$ in (\ref{t11}) are then defined by
\begin{equation}\label{halflambdas}
\lambda_j := \lambda^{L}_{j}+\frac12\in \frac12 +\Z, \,\,\, j=1,\dots, r,
\end{equation}
in accordance with the fact that ${\rm div}\,(dZ)= - D_1 \cdots -D_r$ (see Lemma \ref{lemmacox}).
 Note that $P_L$ is obtained from
the moment polytope $P_{X}$ by
shifts of $\frac12$ along each of the integral primitive inward
pointing normals. (See Remarks \ref{rmk:CP1} and \ref{rmk:CP2} below for examples.)
We will call $P_{L}\subset P_{X}$ the \textit{corrected polytope}.

 We equip $L$ with a $U(1)$
connection $\nabla^{I}$ with curvature $F_{\nabla^{I}}= -i\omega+\frac{i}
{2}\rho_{I}$. Since $H^{1}(X)=0$ this connection is unique up to isomorphism.

Following the reasoning in the last section, and noticing that $\sqrt{|K_I|}$ and $\mu_I$ (see (\ref{mumu}))
exist always even if $\sqrt{K_I}$ does
not, and that the hermitian structure on $\sqrt{|K_I|}$ gives $||\mu_I||_{\sqrt{|K_I|}}=1$, we set

\begin{definition}\label{qhs}
The quantum Hilbert space for the half-form corrected K\"ahler quantization of
$(X,\omega,L,I)$ is defined by
\[
{\mathcal{H}}^{Q}_{I} = \mathcal{B}_{I}^{Q} \otimes\mu_{I},
\]
where
\[
\mathcal{B}_{I}^{Q}=\{s\in\Gamma(L): \nabla^{I}_{\overline{\mathcal{P}}_{I}}.
s=0\}.
\]
The inner product is defined by
\begin{equation}\label{innerproduct}
\langle \sigma\otimes \mu_I ,\sigma' \otimes \mu_I\rangle = \langle \sigma,\sigma'\rangle = \frac{1}{(2\pi)^n}\int_X h^L(\sigma,\sigma') \frac{\omega^n}{n!}.
\end{equation}
\end{definition}

Now fix a choice of symplectic potential $g$ for the complex structure $I$
on $X$. We define the connection $\nabla^{I}$ on $L$ by (using Lemma
\ref{lemma:canonicalconnections} and (\ref{eqn:uncorrected_conn}))
\begin{align}
\Theta_{v}  &  :=\frac{\nabla^{I}\mathbf{1}_{v}^{U(1)}}{\mathbf{1}_{v}^{U(1)}%
}=-i\,{x}_{v}\cdot d{\theta}_{v}+\frac{i}{2}\sum_{k=1}^n d\theta_{v}^{k}+\frac
{i}{4}\left(  \frac{\partial}{\partial x_{v}}\log\det G_{v}\right)  \cdot
G_{v}^{-1}d\theta_{v}\label{eqn:Lconndef_vert}\\
&  =-i\,{x}_{v}\cdot d{\theta}_{v}+\frac{i}{2}\Im\left(  \partial\log\det
G_{v}+\sum_{k=1}^{n}dz_{v}^{k}\right). \nonumber
\end{align}
On the open orbit $\check X$, the connection is then given by
\begin{align}
\Theta_{0}  &  :=-i\,{x}\cdot d{\theta}+\frac{i}{4}\left(  \frac{\partial
}{\partial x}\log\det G\right)  \cdot G^{-1}d\theta\label{eqn:Lconndef_open}\\
&  =-i\,{x}\cdot d{\theta+}\frac{i}{2}\Im\partial\log\det G.\nonumber
\end{align}
One may check that $\Theta_{v}-\Theta_{v^{\prime}}=d\log\tilde{g}_{v^{\prime
}v}^{L}$ and $\Theta_{v}-\Theta_{0}=d\log\tilde{g}_{0v}^{L}$ so that
$\{\Theta_{0},\Theta_{v}: v\in V\}$ does indeed define a $U(1)$-connection on $L$.

\begin{remark}
Note that, even though $d\theta_v^j$ is singular as $x_v^j \to 0$, (\ref{eqn:Lconndef_vert})
defines a non-singular 1-form on $U_v$, as can be verified by studying the
behavior of $G_v$ or using the coordinates $\{a_v^j,b_v^j\}_{j=1,\dots,n}$.
\end{remark}

The complex structure $I$ and the connection $\nabla^{I}$ combine to give a
holomorphic structure on $L$ which we can describe by giving the resulting
$I$-holomorphic sections of $L$.

Let $h_{0}^{I}({x})={x} \cdot\partial g/\partial{x}-g$ and $h_{v}^{I}({x}%
_{v})={x}_{v}\cdot\partial g/\partial{x}_{v}-g$. Also, note that $\det G_v = (\det A_v)^{-2}\det G $.

\begin{lemma}
\label{lemma:holsections}An $I$-holomorphic section of $L$, $s\in\mathcal{B}_{I}^{Q}$,
is locally given by $\left.  s\right\vert _{U_{v}}=s_{v}%
\mathbf{1}_{v}^{U(1)}$ where the function $s_{v}\in C^{\infty}(U_{v})$ is of
the form
\begin{equation}
F_{v}(w_{v})e^{-h_{v}^{I}({x}_{v})}e^{-i\mathbf{1}/2\cdot{\theta}_{v}%
}\left\Vert dZ_{v}\right\Vert _{K_I}^{1/2}. \label{eqn:holsec_vertex}
\end{equation}
On the orbit one then obtains $\left.  s\right\vert _{U_{0}}=s_{0}\mathbf{1}_{0}^{U(1)},$
where the function $s_{0}\in C^{\infty}(U_{0})$ is of the form
\begin{equation}
F_{0}(w)e^{-h_{0}^{I}({x})}\left\Vert dZ\right\Vert _{K_I}^{1/2},
\label{eqn:holsec_open}%
\end{equation}
where $F_{0}$ is holomorphic and $F_{v}(w_{v})=w_{v}^{\lambda_{v}}F_{0}%
(w_{v}^{A_{v}}) |\det A_v|^{\frac12}.$
\end{lemma}

\begin{remark}
  Since the $\lambda_v$ are generally only half-integer, the ``functions'' $F_v$ in the above theorem are not single valued. However, $F_ve^{-i\frac{\mathbf 1}{2}\cdot \theta_v}$ is single valued. In fact, the collection
$\{F_0,F_v: v\in V\}$ defines a ramified section of $L$. There are other possible geometric interpretations of such an object,
for instance through the notion of Kawamata covering, but we will not pursue them here.
\end{remark}

\begin{proof}
We compute first in the open orbit. From
(\ref{partialz}), we see that
\[
-i\,{x}\cdot d{\theta}\left(  \frac{\partial}{\partial\bar{z}^{j}}\right)
=\frac{1}{2}x^{j}.
\]
Recall that ${x}$ can be expressed in terms of ${y}$ via the Legendre
transform ${x}=\partial h_{0}^{I}/\partial{y}$, where $h_{0}^{I}={x}\cdot
{y}-g$. Since $\partial/\partial z^{j}=\frac{1}%
{2}(\partial/\partial y^{j}-i\partial/\partial\theta^{j})$, we have
$x^{j}=2\partial h_{0}^{I}/\partial\bar{z}^{j}$ so that
\begin{equation}
-i\,{x}\cdot d{\theta}\left(  \frac{\partial}{\partial\bar{z}^{j}}\right)
=\frac{\partial h_{0}^{I}}{\partial\bar{z}^{j}}.
\label{eqn:(0,1)-deriv-uncorr}%
\end{equation}

Next, from Lemma \ref{lemma:canonicalconnections} we know that
\begin{align*}
\frac{i}{4} \left(  \frac{\partial}{\partial x} \log\det G\right)
G^{-1}d\theta &  =\frac{i}{2}\Im\partial\log\det G\\
&  =\frac{1}{4}(\partial-\bar{\partial})\log\det G
\end{align*}
so that
\begin{equation}
\frac{i}{4} \left(  \frac{\partial}{\partial x} \log\det G\right)  G^{-1}
d\theta\left(  \frac{\partial}{\partial\bar{z}^{j}}\right)  =-\frac{1}{4}%
\frac{\partial}{\partial\bar{z}^{j}}\log\det G.
\label{eqn:(0,1)-deriv-halfform}%
\end{equation}

Combining (\ref{eqn:(0,1)-deriv-uncorr}) and (\ref{eqn:(0,1)-deriv-halfform}),
we see that a section $s=f\mathbf{1}_{0}^{U(1)}\in\Gamma_{U_{0}}(L)$ is
holomorphic if and only if $f$ satisfies the differential equation
\[
\frac{\partial f}{\partial\bar{z}^{j}}+f\,\frac{\partial}{\partial\bar{z}^{j}
}\left(  h_{0}^{I}-\frac{1}{4}\log\det G\right)  =0,
\]
for each $j=1,\dots,n$. We solve this easily to see that $s$ is holomorphic if
and only if $f$ is of the form
\[
f=F(w)e^{-h_{0}^{I}}(\det G)^{1/4},
\]
where $F$ is an holomorphic function in $U_0$.
From Lemma \ref{lemma:normdz}, we recognize $(\det G)^{1/4}=\left\Vert
dZ\right\Vert _{K}^{1/2}$ to obtain (\ref{eqn:holsec_open}) as desired.

We have computed that $F(w)e^{-h_{0}^{I}}(\det G)^{1/4}\mathbf{1}_{0}^{U(1)}$
is holomorphic. Using the transition functions $\tilde{g}_{v0},$ we conclude
that $F(w)e^{-h_{0}^{I}}(\det G)^{1/4}\tilde{g}_{v0}(w_{v})\mathbf{1}_{v}^{U(1)}$,
when expressed in terms of $w_{v}$, should be holomorphic.
First, note that
\[
x\cdot{y}={x}_{v}\cdot{y}_{v}-\lambda_{v}\cdot{y}_{v}%
\]
which implies $h_{0}^{I}=h_{v}^{I}-\lambda_{v}\cdot{y}_{v}$. From
(\ref{eqn:Ltransfns}) we therefore see that the holomorphic combination in
$U_{v}$ should be
\begin{align*}
F(w)e^{-h_{v}^{I}+\lambda_{v}\cdot{y}_{v}}(\det G)^{1/4}e^{i\lambda_{v}%
^{L}\cdot{\theta}_{v}}\mathbf{1}_{v}^{U(1)}  &  =F(w)e^{-h_{v}^{I}%
}e^{\lambda_{v}\cdot{y}_{v}+i\lambda_{v}\cdot{\theta}_{v}}e^{-i\mathbf{1}%
/2\cdot{\theta}_{v}}(\det G)^{1/4}\mathbf{1}_{v}^{U(1)}\\
&  =F(w)w_{v}^{\lambda_{v}}e^{-h_{v}^{I}}e^{-i\mathbf{1}/2\cdot{\theta}_{v}
}|\det A_v|^{\frac12}\left\Vert dZ_{v}\right\Vert _{K_I}^{1/2}\mathbf{1}_{v}^{U(1)}.
\end{align*}
By (\ref{eqn:ww_v}) we see that if we set $F_{v}(w_{v}):=F(w_{v}^{A_{v}}%
)w_{v}^{\lambda_{v}}|\det A_v|^{\frac12}$, we obtain (\ref{eqn:holsec_vertex}) and the final
statement of the lemma as desired.
\end{proof}

\bigskip

\begin{remark}
We can rewrite a local holomorphic section on $U_{v}\cap U_0$ in a way similar to that
in \cite{Baier-Florentino-Mourao-Nunes} as follows:%
\begin{align*}
F_{v}(w_{v})e^{-h_{v}^{I}({x}_{v})}e^{-\mathbf{1}/2\cdot{\theta}_{v}%
}\left\Vert dZ_{v}\right\Vert _{K_I}^{1/2}  &  =F_{0}(w_{v}^{A_{v}}%
)w_{v}^{\lambda_{v}}e^{-h_{v}^{I}({x}_{v})}e^{-i\mathbf{1}/2\cdot{\theta}_{v}%
}|\det A_v|^{\frac12}\left\Vert dZ_{v}\right\Vert _{K_I}^{1/2}\\
&  =F_{0}e^{-h_{v}^{I}+\lambda_{v}\cdot{y}_{v}}e^{i\lambda_{v}^{L}\cdot
{\theta}_{v}}|\det A_v|^{\frac12} \left\Vert dZ_{v}\right\Vert _{K_I}^{1/2}.
\end{align*}
The combination $h_{v}^{I}-\lambda_{v}\cdot{y_v}$ corresponds to $h_{m}$ (for
$m=v$) in \cite{Baier-Florentino-Mourao-Nunes}.
\end{remark}

\begin{theorem}
\label{cor:holsections} The Hilbert space $\mathcal{B}_{I}^{Q}$ of holomorphic sections of $L$ has an orthogonal
basis $\{\sigma^{m}\}_{m\in P_{L}\cap
\mathbb{Z}^{n}}$ where $\sigma^{m}$ is locally given by
\begin{align*}
\sigma_{0}^{m}  &  =w^{m}e^{-h_{0}^{I}}\left\Vert dZ\right\Vert _{K_I}%
^{1/2}\mathbf{1}_{0}^{U(1)}\text{, and}\\
\sigma_{v}^{m}  &  =w_{v}^{A_{v}m+\lambda_{v}^{L}}e^{-h_{v}^{I}+\mathbf{1}%
/2\cdot{y}_{v}}|\det A_v|^{\frac12}\left\Vert dZ_{v}\right\Vert _{K_I}^{1/2}\mathbf{1}_{v}^{U(1)},
\end{align*}
over the open orbit and holomorphic vertex charts, respectively. The corresponding orthogonal basis for the quantum Hilbert
space ${\cal H}^Q_I$ is given by
$\{\hat\sigma^m:=\sigma^m\otimes \mu_I\}_{m\in P_{L}\cap
\mathbb{Z}^{n}}$.
\end{theorem}

\begin{proof}
From Lemma \ref{lemma:holsections} we can certainly find a basis for the space
of holomorphic sections of $L$ consisting of elements given locally over the
open orbit by
\[
w^{m}e^{-h_{0}^{I}}\left\Vert dZ\right\Vert _{K_I}^{1/2}\mathbf{1}_{0}^{U(1)},
\]
where $m\in\mathbb{Z}^{n}$, with the corresponding expressions over the holomorphic vertex
charts. Such a section will have poles unless $m$
belongs to the corrected polytope $P_{L}$. The fact that $\sigma^m$ and $\sigma^{m'}$ are orthogonal for
$m\neq m'$ follows immediately from integration along $\mathbb{T}^n$.
\end{proof}

\bigskip
Therefore, the space of half-form corrected holomorphic wave functions for the
K\"ahler quantization of $X$ has a natural basis whose elements are labeled by
the integral points of the corrected polytope $P_{L}$. These coincide also with the
(interior) integral points of the moment polytope $P_{X}$ and
they correspond to shifted nonsingular Bohr--Sommerfeld fibers of $(X,\omega)$.

\begin{remark}\label{rmk:CP1}
Pictured below is the moment polytope $P_X$ for $X=\mathbb{C}\mathbb{P}^{1}$, in the case
$[\frac{\omega}{2\pi}]=\frac32 c_1(\mathbb{C}\mathbb{P}^{1})=3c_1({\cal O}(1))$. On the left, we show the more
standard choice of moment polytope, with integral vertices. On the right, we show the moment
polytope chosen in accordance with (\ref{t11}) and (\ref{halflambdas}),
that is such that $\lambda_1,\lambda_2$ are half-integral. In this example,
$L\cong {\cal O}(2)$, the corrected polytope is $P_L=[0,2]$ and the moment polytope
is $P_X=[-\frac12,\frac32]$. One has $\dim {\cal H}^Q=3.$
\begin{center}
\begin{pspicture*}(-50pt,-30pt)(65pt,15pt) \psset{unit=20pt,xunit=15pt,yunit=15pt}
\pcline{|-|}(-2,0)(4,0)
\psdots(-2,0)(0,0)(2,0)(4,0)
\rput(-2,-1){$0$}
\rput(0,-1){$1$}
\rput(2,-1){$2$}
\rput(4,-1){$3$}
\end{pspicture*}
$\qquad$
\begin{pspicture*}(-50pt,-30pt)(65pt,15pt) \psset{unit=20pt,xunit=15pt,yunit=15pt}
\pcline{|-|}(-2,0)(4,0)
\psdots(-1,0)(1,0)(3,0)
\rput(-2,-1){$-\tfrac{1}{2}$}
\rput(-1,-1){$0$}
\rput(1,-1){$1$}
\rput(3,-1){$2$}
\rput(4,-1){$\tfrac{5}{2}$}
\end{pspicture*}
\end{center}
\end{remark}

\begin{remark}\label{rmk:CP2}
According to Theorem \ref{cor:holsections}, we can count holomorphic sections
of $L$ by counting integral points {\it inside} the moment polytope $P_X$,
which are exactly the integral
points which occur in the corrected polytope $P_{L}$. Pictured below is one of such
polytopes when $X=\mathbb{C}\mathbb{P}^{2}\#\overline{\mathbb{C}\mathbb{P}}^{2}$, that is,
$\mathbb{C}\mathbb{P}^{2}$ blown up at a point. \bigskip
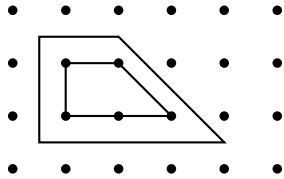
\begin{center}
\begin{pspicture*}(-45pt,-25pt)(65pt,65pt) \psset{unit=20pt,xunit=20pt,yunit=20pt} \pspolygon[linestyle=solid,fillstyle=none](-1,1)(-1,2)(0,2)(1,1) \pspolygon[linestyle=solid,fillstyle=none](-1.5,0.5)(-1.5,2.5)(0,2.5)(2,0.5)
\psdots(-2,0)(-1,0)(0,0)(1,0)(2,0)(3,0)(-2,1)(-1,1)(0,1)(1,1)(2,1)(3,1)(-2,2)(-1,2)(0,2)(1,2)(2,2)(3,2)(-2,3)(-1,3)(0,3)(1,3)(2,3)(3,3) \end{pspicture*}
\end{center}
\end{remark}

\begin{remark}
One can consider the case $[\omega/2\pi]=c_{1}(X)/2,$ where
$L= {\mathcal O}_X$ is just the structure sheaf of $X$ (which is,
of course, not ample). In this case, there is only one
integral point inside the polytope $P_{X}$ corresponding to the constant
function ${1}\in H^0({\mathcal O}_X)$. As we will see in Section \ref{sec:families}, when we study degenerations of
the complex structure, also in this case we have convergence to a Dirac delta distribution supported
on the (shifted) Bohr--Sommerfeld fiber above that integral point.
\end{remark}

\section{Half-form corrected  quantization in the singular real toric polarization}
\label{sec:real}

\subsection{Distributional sections}

\label{subsec:distsect}

In order to study quantization in the real toric polarization of $(X,\omega)$,
following \cite{Baier-Florentino-Mourao-Nunes},
we consider distributional sections of $L$. Let us briefly recall how one can
define covariant differentiation in this case. Let $L^2(L)$ denote the Hilbert
space of $L^2$ sections of $L$. Consider the rigged Hilbert space (see Sections 4.2 and 4.3 of \cite{GV})
$(\Gamma(L), L^2(L),\Gamma(\bar{L})^{\prime})$, where
$\Gamma(\bar{L})^{\prime}$ is the space of distributional sections of $L$ given by the topological dual
of $\Gamma(\bar L)$. One has the continuous inclusions
\[
\Gamma(L)\subset L^{2}(L)\subset\Gamma(\bar{L})^{\prime},
\]
where we embed $\sigma\in\Gamma(L)\mapsto i(\sigma)\in\Gamma(\bar{L})^{\prime
}$ via the Liouville volume form; i.e.
\[
i(\sigma)(\bar{\tau}):=\frac{1}{(2\pi)^{n}}\int_{X}h^{L}(\sigma,\tau
)\frac{\omega^{n}}{n!}.
\]
(In particular, we may view $I$-holomorphic sections of $L$ as
distributional sections.) We have then, for $\sigma\in\Gamma(L)$,
\[
i(\sigma)(\bar{\tau})=\langle\sigma,\tau\rangle_{L}, \,\,\forall\tau\in
\Gamma(L).
\]

Let $\nabla^{L}$ be a connection on $L$ and let $\nabla^{*}$ be the adjoint of
the (unbounded) operator $\nabla^{L}$ on the Hilbert space $L^{2}(X,L)$, so
that
\[
i(\nabla^{L}\sigma)(\bar\tau) = \langle\nabla^{L} \sigma,\tau\rangle_{L} =
\langle\sigma,\nabla^{*}\tau\rangle_{L},\,\, \forall\sigma,\tau\in\Gamma(L).
\]

We can now define covariant differentiation of distributional sections, which
we will still denote by $\nabla^{L}$, by
\[
\nabla^{L} (\sigma) (\bar\tau) = \sigma(\overline{\nabla^{*} \tau}),\,\,
\sigma\in\Gamma(\bar L)^{\prime}, \tau\in\Gamma(L)
\]
so that, as distributions
\[
\nabla^{L} (i\sigma)= i(\nabla^{L} \sigma), \,\, \forall\sigma\in\Gamma(L).
\]

In the next sections, we will interpret holomorphic sections as distributional sections in this way, and we
will identify $\sigma\in\Gamma(L)$ with $i(\sigma)\in \Gamma(\bar L)'$.

\subsection{Quantization in the (singular) real
polarization\label{subsec:realQuant}}

Recall that the real singular toric polarization is defined by
\[
\mathcal{P}_{\mathbb{R}}(p)=\operatorname*{span}\nolimits_{\mathbb{C}
}\left\{  \left(\frac{\partial}{\partial\theta^{i}}\right)_p, i=1,\dots,n\right\},\,\, \forall p\in X  .
\]
In this section, we will define the
half-form corrected quantization of $X$ in this polarization directly in terms
of covariantly constant sections. Recall
the families of toric complex structures considered in
\cite{Baier-Florentino-Mourao-Nunes}. For any smooth function $\psi$ which is strictly convex
on a neighborhood of $P_{X}$, for any  $\varphi\in C_{P_{X}}^{\infty}(P_X)$ and for any
$s\in\mathbb{R}_{\geq 0}$, the sum $\varphi +s\psi$ is in $C_{P_{X}}^{\infty}(P_X)$ and
hence defines a K\"{a}hler
structure on $X$ with symplectic potential
\[
g_{s}:=g_{P_{X}}+\varphi+s\psi.
\]
Denote the corresponding $s$-dependent complex structure by $I_{s}$.

If ${\mathcal P}$ is a polarization of $(X,\omega)$, denote by $C^\infty({\mathcal P})$
its space of smooth sections,
$$C^\infty({\mathcal P})=\{\xi\in C^\infty(TX \otimes \C): \xi (p)\in {\mathcal P}_p\},$$
where $C^\infty(TX \otimes \C)$ is the space of smooth sections of the complexified tangent bundle of $X$.

\begin{theorem}(\cite{Baier-Florentino-Mourao-Nunes}, Theorem 1.2 p. 415, Theorem 3.4, p. 429)\\
Pointwise on the dense open orbit $\check X$, as vector fields,
$$
\frac{\partial}{\partial \bar z^j_s} = \frac12 \left(\frac{\partial}{\partial y_s^j} +i \frac{\partial}
{\partial \theta^j}\right) \rightarrow
\frac{i}2 \frac{\partial}{\partial \theta^j} , \,\, {\rm as}\,\, s\to\infty.
$$
Therefore, at each point $p\in \check X$, the holomorphic polarizations
$\mathcal{P}_{s}
,~s\geq 0$ of $X$,  associated to the complex structures $I_{s}$,
converge, as $s\rightarrow\infty$, to the real toric
polarization, in the Lagrangian Grassmannian of $T_pX \otimes \C$ and
\begin{equation}\label{polar}
C^{\infty}(\lim_{s\rightarrow\infty}\mathcal{P}_{s})=C^{\infty}(\mathcal{P}_{\mathbb{R}}).
\end{equation}
\end{theorem}

\begin{remark}
The equality in (\ref{polar}) is an equality of spaces of smooth sections of two different
polarizations $\lim_{s\rightarrow\infty}\mathcal{P}_{s}$  and $\mathcal{P}_{\mathbb{R}}$
which coincide over $\check X$ but not over $X\setminus \check X$.
(See Theorem 3.4 of \cite{Baier-Florentino-Mourao-Nunes}).
\end{remark}

{}From the expressions for the half-form corrected connection in
(\ref{eqn:Lconndef_vert}), we see that in the local trivializations $\mathbf{1}_v^{U(1)}$
\begin{equation}
-2i \nabla^{I_{s}}_{{\partial/\partial\bar z_{s_{v}}^{j}}} \rightarrow
\nabla^\R_{{\partial}/{\partial \theta^j_v}}:=\frac{\partial}{\partial \theta_v^j}-ix_{v}^{j}+\frac{i}{2} ,\text{ as }s\rightarrow\infty,
\label{limitconnvert}
\end{equation}
in the sense that
\begin{equation}\nonumber
-2i \left( \nabla_{{\partial}/{\partial\bar z_{s_{v}}^{j}}}^{I_{s}} \sigma \right)(\bar\tau)\rightarrow
\left(\nabla^\R_{{\partial}/{\partial \theta^j_v}}\sigma\right)(\tau) ,\text{ as }s\rightarrow\infty,
\end{equation}
$\forall\sigma \in \Gamma(\bar L)'$, $\forall \tau\in\Gamma(L)$.
Similarly, from (\ref{eqn:Lconndef_open}), on the open orbit in the trivialization $\mathbf{1}_0^{U(1)}$ we have
\begin{equation}
-2i\nabla^{I_{s}}_{{\partial/\partial\bar z_{s}^{j}}}\rightarrow \nabla^\R_{{\partial/\partial
\theta^{j}}}=\frac{\partial}{\partial \theta^j}-ix^{j},\text{ as }s\rightarrow\infty. \label{limitconnopen}
\end{equation}

We will take the expressions on the right hand side of (\ref{limitconnvert})
and (\ref{limitconnopen}) to define a partial connection $\nabla^{\mathbb{R}}$
on $\Gamma(\bar L)'$, along $\mathcal{P}_{\mathbb{R}},$ which will be used to define the
quantization in this polarization.
Let
\begin{equation}\label{imaculada}
\mathcal{B}_{\mathbb{R}}^{Q}=\ker\nabla^{\mathbb{R}}=
\bigcap_{j=1}^n \ker \nabla^\R_{\partial/\partial \theta^j}
\subset\Gamma(\bar L)^{\prime}.
\end{equation}

\begin{remark}
We note that additive term $\frac{i}{2}$ in the right hand side of
(\ref{limitconnvert}), corresponds to a
limiting Chern connection on $K_{I_s}$ which is flat on $U_0$ and singular along
$\cup_{i=1}^r D_i$. The addition of this singular connection to the prequantum connection is
at the core of our approach to the half-form quantization in the singular real toric polarization.
We will describe in more detail the singular behavior of the limiting connection at the end of this section.
\end{remark}

\begin{remark}
As explained in the previous section, we should think of $L$ as the tensor
product of the uncorrected bundle $\ell$ with the smooth bundle with $U(1)$-valued transition functions
$\widetilde{\sqrt{K}}^{U(1)},$ though in general these may not exist individually. On the
other hand, the geometric quantization associated to a polarization
$\mathcal{P}$ is supposed to be the space of $\mathcal{P}$-covariantly
constant sections of the tensor product of the uncorrected bundle with the
square root of the canonical bundle associated $\mathcal{P}$. For the real
polarization $\mathcal{P}_{\mathbb{R}}$, the sections of the associated
canonical bundle are $n$-forms of the form $a(x) dx^{1}\wedge\cdots\wedge dx^{n}$.
Of course, the $U(1)$-part of this is hidden in the sections of $L$, and what is
missing is the modulus of the square root of the canonical bundle associated
to $\mathcal{P}_{\mathbb{R}}.$

To put it another way, according to the standard procedures of geometric
quantization, we should actually define the quantization $\mathcal{H}%
_{\mathbb{R}}^{Q}$ to be sections of $L\otimes\sqrt{\left
\vert{K_{\mathcal{P}_{\mathbb{R}}}}\right\vert}$. Some care must
be taken to interpret exactly what is meant by $\left\vert K_{\mathcal{P}%
_{\mathbb{R}}}\right\vert $ (and hence what is meant by its square root). In
the next section, we will see that sections of $\sqrt{\left\vert
{K_{\mathcal{P}_{\mathbb{R}}}}\right\vert} $ can be thought of as maps
on the space of $n$-tuples of vector fields. Then,
$\sqrt{\left\vert{K_{\mathcal{P}_{\mathbb{R}}}}\right\vert}$ admits a
canonical section $dX:=dx^{1}\wedge\cdots\wedge dx^{n}$ on $U_{0}$, and $0$
otherwise. Hence, we should define $\mathcal{H}_{\mathbb{R}}^{Q}=\mathcal{B}_{\mathbb{R}}^{Q}\otimes
\sqrt{\left\vert dX\right\vert}$. Of course, at this point, such a change is merely
cosmetic. On the other hand, we will see in the next section that such
expressions arise naturally when studying the degenerations of the complex
structure on $X$ to the real polarization at the level of holomorphic sections.
\end{remark}

\begin{definition}\label{fmi}
The vector space of quantum states for the half-form corrected quantization
of $(X,\omega,L)$ in the toric polarization $\mathcal{P}_{\mathbb{R}}$ is
defined by
$$
\mathcal{H}_{\mathbb{R}}^{Q}=\mathcal{B}_{\mathbb{R}}^{Q} \otimes\sqrt{\left\vert dX\right\vert},
$$
where $\mathcal{B}_{\mathbb{R}}^{Q}$ was defined in (\ref{imaculada}).
\end{definition}

\begin{remark}
A natural Hilbert space structure in ${\cal H}^Q_\R$ will be introduced in Section \ref{sec:families},
via degeneration of K\"ahler quantizations of $(X,\omega,L,I)$.
\end{remark}

The open orbit $U_{0}$ carries a free $\mathbb{T}^{n}$-action which
lifts to $L_{U_{0}}$ via geometric quantization. This action is generated by
$$
\nabla_{\frac{\partial}{\partial \theta}} + ix.
$$
Then, the trivializing section $\mathbf{1}_{0}^{U(1)}$ is $\mathbb{T}^{n}$-invariant and
if $\tau\in\Gamma_{U_{0}}(L)$ is given by $\tau=\tau_{0}\mathbf{1}_{0}^{U(1)}$ for a smooth function
$\tau_{0}\in C^{\infty}(U_{0})$, we can decompose it into Fourier
modes with respect to the $\mathbb{T}^{n}$-action. Specifically,
\[
\tau_0({x,\theta})=\sum_{m\in \Z^n}e^{-im\cdot{\theta}}\hat{\tau}_{0,m}({x}),
\]
where $\hat{\tau}_{0,m}({x)=}\frac{1}{(2\pi)^{n}}\int_{\mathbb{T}^{n}
}e^{im\cdot{\theta}}\tau_0({x,\theta})d{\theta}$ is the $m$-th Fourier mode
of $\tau_0.$

For $m\in P_{L}\cap{\mathbb{Z}}^{n}$, let $\delta^{m}\in\Gamma(\bar L)^{\prime}$ be
the distributional section defined by
\begin{equation}\label{deltam}
\delta^{m} (\bar \tau) = \bar{\hat{\tau}}_{m}(m) = \frac{1}{(2\pi)^{n}}\int_{\mathbb{T}^{n}
}e^{im\cdot{\theta}}\bar\tau_0({m,\theta})d{\theta},
\end{equation}
for all $\tau\in\Gamma(L)$. We have

\begin{theorem}
\label{thm:realQ} The vector space $\mathcal{H}_{\mathbb{R}}^{Q}$ is the
finite-dimensional vector space generated by \newline
$\{\delta^{m}\otimes \sqrt{|dX|}\}_{m\in P_{X}\cap{\mathbb{Z}}^{n}}.$
\end{theorem}

Therefore, quantization in the real polarization is also given by the integral
points in the interior of the moment polytope.
Recall, from Section \ref{corrq} and Remarks \ref{rmk:CP1} and \ref{rmk:CP2},
that $P_L\subset P_X$ is obtained from $P_X$ by a one-half shift along the inward
pointing normals to the facets of $P_X$.
In Section \ref{sec:families},
these distributional sections will be described as coming from holomorphic
sections through degeneration of the complex structure.

\bigskip

\begin{proof}
[Proof of Theorem \ref{thm:realQ}]By acting with $\nabla^{\mathbb{R}}_{{\partial}/{\partial \theta^j}},
j=1,\dots,n$, in (\ref{limitconnopen}) on the distributions
$\delta^{m}$ in (\ref{deltam}) we conclude that they belong to ${\rm ker}\,\nabla^\R$. Moreover, any element of the
kernel can be restricted to the open orbit, by restricting it to sections of
$\bar L$ with compact support contained in the open orbit. From, Proposition 3.1 in
\cite{Baier-Florentino-Mourao-Nunes}, we see that such restrictions can
only have support along $\mu^{-1}(\check{P}_{X}\cap \Z^n)$, as these are the only fibers
along which $\nabla^\R$ has trivial holonomy, and one easily
verifies that $\delta^{m}$ is the unique (up to a constant) solution supported
on $\mu^{-1}(m)$.

It remains to be shown that there are no more elements in the kernel of
$\nabla^{\mathbb{R}}$. All we need to show is that there are no solutions with
support along $\mu^{-1}(\partial P_{X}).$ Let us consider a solution with
support along $\mu^{-1}(x_{v}^{j}=0)$, for some fixed $j=1,\dots,n$.
Let $\check x_{v}= (x_{v}^{1},\dots,x_{v}^{j-1},x_{v}^{j+1},\dots,x_{v}^{n})$ and
$\check \theta_v =(\theta_v^1,\dots,\theta_v^{j-1},\theta_v^{j+1},\dots,\theta_v^n)$.
In a neighborhood of the preimage by $\mu$ of the interior of the facet $x_{v}^{j}=0$ of $P_X$,
we can take coordinates $(u,v,\check x_v,\check \theta_v)$ (see, for example
\cite{DP09}, \cite{Baier-Florentino-Mourao-Nunes}), so that
$x_v^j =0 \Leftrightarrow (u,v)=(0,0)$ and
$$
\nabla^\R_{\frac{\partial}{\partial \theta_v^j}} = -i\left(-v\frac{\partial}{\partial u}+u\frac{\partial}{\partial v}\right)
+\frac{i}{2},
$$
in that neighborhood. A solution with support along the facet will be of the form,
$$
\sigma = \sum_{k,l=0}^\infty \alpha_{kl}(\check x_v, \check \theta_v) \delta^{(k)}(u)\delta^{(l)}(v)
\mathbf{1}_v^{U(1)},
$$
where only a finite number of terms
in the sum can be nonzero and where $\delta^{(k)}$ denotes the order-$k$
derivative of the Dirac $\delta$ distribution. (See
Theorem 2.3.5 of \cite{Hormander}.) Using
$x\delta^{(k)}(x)=-k\delta^{(k-1)}(x)$ and polynomial test sections of the form
$u^kv^l\chi \mathbf{1}_v^{U(1)}$, where $\chi$ is a cutoff
function which is constant and equal to 1 in the neighborhood, t
he condition
$\nabla^\R_{\frac{\partial}{\partial \theta_v^j}}\sigma =0$ then implies that such a distributional section
is zero.
Therefore, no nonzero solutions
of this form exist.
\end{proof}

\bigskip

As we will see in the next section, there is complete agreement between the direct approach to
half-form corrected quantization in the real polarization and the approach
based on degeneration of holomorphic sections. We note that the
partial connection $\nabla^{\mathbb{R}}$ \textquotedblleft
remembers\textquotedblright\ the degeneration procedure due to the
contribution of the Ricci-curvature term.

It is possible to gain some more geometric intuition about the fact that
the boundary of $P_X$ does not contribute to the kernel of
$\nabla^{\mathbb{R}}$, unlike what happens without the half-form correction \cite{Baier-Florentino-Mourao-Nunes}.
As $s\rightarrow\infty$, the piece of the connection
$\nabla^{I_s}$ coming from the Levi--Civita connection on the canonical bundle
of $X$ develops curvature singularities outside of the open orbit $U_{0}.$
This is behind the fact, shown above, that for the real polarization the half-form
correction forbids solutions of the covariant constancy equations supported on $\mu^{-1}(\partial P_{X})$.
For completeness, let us describe these curvature
singularities in more detail.

Recall Lemma \ref{lemma:canonicalconnections}, which gives explicit
expressions for the Chern connection $1$-forms on $K^{U(1)}$ (induced from the
Chern connection on $K$) in the vertex atlas. The proof of the following Proposition is immediate.
\begin{proposition}
$$
\lim_{s\to\infty} \Theta^{K_{I_s}}_v = \sum_{j=1}^n d\theta_v^j,
$$
in the sense that for any closed curve $C$, with  $C\subset U_v\setminus {\mu^{-1}(\partial P_X)}=U_0\cong
(\C^*)^n$, the holonomy of the singular connection, along $C$ depends only on the homotopy class
of $C$ in $U_0$ and,
$$
\lim_{s\to\infty}
\oint_C \Theta^{K_{I_s}}_v = {i} \oint_C \sum_{j=1}^n d\theta_v^j.
$$
\end{proposition}

Therefore, in the limit $s\to\infty$, we obtain a singular connection on $K^{U(1)}$, flat on
$\mu^{-1}(\check P_X)$, with curvature supported on $\mu^{-1}(\partial P_X) = \cup_{i=1}^r D_i$ and
with nonvanishing monodromies around the toric invariant divisors.
In the vertex chart $U_v$ the curvature, in the limit $s\to\infty$, is given by
the following current
$$
2\pi i \sum_{j=1}^n \delta(a_v^j)\delta(b_v^j) da_v^j\wedge db_v^j,
$$
where $(a_v^1,b_v^1,\dots,a_v^n,b_v^n)$, with
$a_v^j=\sqrt{x_v^j}\,\cos \theta_v^j, b_v^j= \sqrt{x_v^j}\,\sin \theta_v^j$, $i=1,\dots,n$,
are coordinates on $U_v$.

\subsection{Degeneration to the real polarization\label{sec:families}}

In this section, we will obtain the degeneration, as $s\to\infty$, of the (appropriately $L^2$-normalized) elements $\hat \sigma^m_s$
of the orthogonal basis of
${\cal H}^Q_{I_s}$, defined in Theorem \ref{cor:holsections}, to the same
distributional sections $\delta^m\otimes \sqrt{|dX|}\in {\cal H}^Q_\R$ obtained in the previous section, see Definition \ref{fmi},
(\ref{deltam}) and Theorem \ref{thm:realQ}. In particular, this will allow us to define a natural inner product in ${\cal H}^Q_\R.$
We will first study the degeneration of the basis elements
$\sigma^m_s$ of ${\cal B}^Q_{I_s}$ and then the degeneration of the sections $\mu_{I_s}$ of $\sqrt{|K_{I_s}|}$.

As the complex structure $I_{s}$ varies with $s$, we can regard the spaces ${\cal{B}}^Q_{I_s}$ of
$I_s$-holomorphic sections, described in Definition \ref{qhs},
as finite-dimensional subspaces of the fixed infinite-dimensional space of distributional
sections of $L$ (see Section \ref{subsec:distsect}). That is, for all $s$,
\[
{\cal{B}}^Q_{I_s}\subset\Gamma(L)\subset \Gamma(\bar L)'.
\]
We will study (weak) convergence of $I_s$-holomorphic sections in $\Gamma(\bar L)'$, as $s\to \infty$.
Denote the $I_s$-holomorphic section of $L$ associated to $m\in P_{L}\cap \Z^n$ by $\sigma_{s}^{m}$, as
in Theorem \ref{cor:holsections}.
Suppose $\tau\in\Gamma(L)$ is given
locally by $\{\tau_{0}\mathbf{1}_{0}^{U(1)},\tau_{v}\mathbf{1}_{v}^{U(1)}:v\in V\}$
and let $m\in P_{L}\cap\Z^n.$ Then since $\check X$ and the
holomorphic vertex charts $U_{v}$ are dense in $X$, we see from Theorem
\ref{cor:holsections} that
\begin{align}
i(\sigma_{s}^{m})(\bar{\tau})  &  =\frac{1}{(2\pi)^{n}}\int_{\check X}%
w^{m}e^{-h_{0}^{I_{s}}}\left\Vert dZ\right\Vert _{K}^{1/2}\bar{\tau}%
_{0}\,\frac{\omega^{n}}{n!}\label{eqn:sigmamdist}\\
&  =\frac{1}{(2\pi)^{n}}\int_{U_{v}}w_{v}^{\ell_{v}(m)\cdot{y}_{s_{v}}
-h_{v}^{I_{s}} +\mathbf{1}/2\cdot{y}_{s_{v}}}|\det A_v|^{\frac12}  \left\Vert dZ_{v}\right\Vert
_{K}^{1/2}\bar{\tau}_{v}\,\frac{\omega^{n}}{n!},\nonumber
\end{align}
where ${y}_{s_{v}}=\partial g_{s}/\partial{x}_{v}$.

The next lemma, which we recall from \cite{Baier-Florentino-Mourao-Nunes},
will allow us to use Laplace's approximation to compute the
asymptotics that we are interested in.
\begin{lemma}
\cite[Lemma 5.1]{Baier-Florentino-Mourao-Nunes}\label{lemma:fm} For $m\in P_{X}$
and any smooth function $\psi$ which is strictly
convex on a neighborhood of $P_{X},$ let
\[
f_{m}:=({x-}m)\cdot\partial\psi/\partial{x}-\psi.
\]
Then $f_{m}$ has a minimum value of $-\psi(m)$ on $P$ which is obtained at the
unique point ${x}=m.$ Moreover,
\[
\left(  \operatorname{Hess}f_{m}\right)  (m)=\left(  \operatorname{Hess}%
\psi\right)  (m).
\]

\end{lemma}

\begin{remark}
It is important to observe that the function $f_{m}$ in Lemma \ref{lemma:fm}
has a unique minimum on the entire polytope $P_X$, not just the interior, which
implies that the leading order asymptotics that we will be interested in all
arise from the behavior of the integrand at ${x}=m.$
\end{remark}

Recall,

\begin{lemma}
\label{lemma:Laplace}(Laplace's Approximation) Suppose a function $f\in
C^{2}(R)$ on the closed region $R\subset\mathbb{R}^{n}$ has a unique nondegenerate
minimum the unique point ${x}_{0}\in\check{R}$ in the interior of $R$; so in
particular, $\operatorname{Hess}_{{x}_{0}}f$ is positive definite. Then if
$g_{s}$ is a continuous function on $R$ such that $g_{s}\sim s^{r}%
g_{0}+O(s^{r-1}),~s\rightarrow\infty$, we have
\[
\int_{R}e^{-sf}g_{s}\,d{x}\sim\left(  \frac{2\pi}{s}\right)  ^{n/2}%
\frac{e^{-sf({x}_{0})}s^{r}g_{0}({x}_{0})}{\sqrt{\det\left(
\operatorname{Hess}f\right)  ({x}_{0})}},~s\rightarrow\infty.
\]

\end{lemma}

\begin{lemma}
\label{lemma:L2asymp}As $s\rightarrow\infty$, the leading order asymptotic
value of the $L^{2}$-norm of the family of sections $\sigma_{s}^{m}%
,s\in\mathbb{R}_{>0}$ is%
\[
\left\Vert \sigma_{s}^{m}\right\Vert _{L^{2}}^{2}\sim\pi^{n/2}e^{2g_{s}(m)}.
\]

\end{lemma}

\begin{proof}
To compute the asymptotics we can restrict the integral to the open orbit,
where we have
\begin{align}
\left\Vert \sigma_{s}^{m}\right\Vert _{L^{2}}^{2}  &  =\frac{1}{(2\pi)^{n}%
}\int_{\check{P}_X\times\mathbb{T}^{n}}\left\vert w^{m}\right\vert
^{2}e^{-2h_{0}^{I_{s}}}\left\Vert dZ\right\Vert _{K}\,\frac{\omega^{n}}%
{n!}\nonumber\\
&  =\frac{1}{(2\pi)^{n}}\int_{\check{P}_X\times\mathbb{T}^{n}}e^{2m\cdot{y}%
_{s}-2({x}\cdot{y}_{s}-g_{s})}(\det G_{s})^{1/2}\,\frac{\omega^{n}}%
{n!}\nonumber\\
&  =\int_{\check{P}_X}e^{-2s(({x}-m)\cdot\partial\psi/\partial{x}-\psi
)}e^{-2(({x}-m)\cdot{y}_{0}-g_{0})}(\det G_{s})^{1/2}d{x} \label{eqn:normint},
\end{align}
where in the last line we used the fact that ${y}_{s}=\partial g_{s}%
/\partial{x=}\partial g_{0}/\partial{x}+s\partial\psi/\partial{x}$.
Note that $G_{s}=G_{0}+s\operatorname{Hess}\,\psi$ implies
\begin{equation}
\det G_{s}\sim s^{n}\det\operatorname{Hess}\,\psi+O(s^{n-1}%
).\label{eqn:detGslim}%
\end{equation}
We would like to apply Laplace's approximation to the integral
(\ref{eqn:normint}). By Lemma \ref{lemma:fm}, the first exponential has the
correct behavior. The only remaining subtlety is to show that the remainder of
the integrand is continuous on $P_X$, which is not immediate since $g_{s}$ is
singular along $\partial P_X$.
Using the explicit expression (\ref{eqn:gP}) and the regularity conditions
(\ref{eqn:greg}), we conclude that $e^{-2(({x}-m)\cdot{y}_{0}-g_{0})}(\det G_{s})^{1/2}$ behaves like
\[
\Pi_{i=1}^{r} \ell_{i}(x) ^{(\frac12 \ell_{i}(m)-\frac12)}%
\]
times a smooth function on $P_X$. Therefore, it is continuous and goes to zero
at the boundary of $P_X$ precisely when $m$ belongs to the corrected polytope
$P_{L}\subset P_{X}$.
From another point of view, the integrand in (\ref{eqn:normint}) is the
pointwise norm of the $\mathbb{T}^{n}$-invariant holomorphic section
$\sigma_{s}^{m}$, which is necessarily continuous on $P_X$.
Then, Laplace's approximation (Lemma \ref{lemma:Laplace}) yields
\[
\left\Vert \sigma_{s}^{m}\right\Vert _{L^{2}}^{2}\sim\left(  \frac{2\pi}%
{s}\right)  ^{n/2}\frac{e^{2s\psi(m)}e^{2g_{0}(m)}s^{n/2}\sqrt{\det
\operatorname{Hess}\psi(m)}}{\sqrt{2^{n}\det\operatorname{Hess}\psi(m)}}%
=\pi^{n/2}e^{2g_{s}(m)}%
\]
as desired.
\end{proof}

\bigskip

Recall that $G_{s}=G_{0}+s\,\operatorname{Hess}\psi$, where $G_{0}%
=\operatorname{Hess}(g_{P}+\varphi)$.

\begin{theorem}
\label{thm:sec_asymps}For each $\tau\in\Gamma(L)$, the leading order
asymptotic value of $i\left(  \sigma_{s}^{m}/\left\Vert \sigma_{s}%
^{m}\right\Vert _{L^{2}}\right)  $ on $\bar{\tau}$ as $s\rightarrow\infty$ is
determined by%
\[
i\left(  \frac{\sigma_{s}^{m}(\det G_{s})^{1/4}}{\left\Vert \sigma_{s}%
^{m}\right\Vert _{L^{2}}}\right)  (\bar{\tau})\sim2^{n/2}\pi^{n/4}\hat
{\bar{\tau}}_{0,m}(m)\text{.}%
\]
That is, in terms of the distributional sections $\delta^{m},~m\in P_{X}\cap\Z^n$,
described in the Section \ref{subsec:realQuant}, formula (\ref{deltam}),
\[
\lim_{s\rightarrow\infty}  \frac{\sigma_{s}^{m}(\det G_{s})^{1/4}%
}{\left\Vert \sigma_{s}^{m}\right\Vert _{L^{2}}}  =2^{n/2}\pi
^{n/4}\delta^{m}.
\]
\end{theorem}

\begin{proof}
We will first compute the asymptotics of $i(\sigma_{s}^{m})(\bar{\tau})$, and
then simply divide by the results of the previous lemma to obtain the desired
expressions. To this end, compute first in the open orbit. From
(\ref{eqn:sigmamdist}) we have
\begin{align*}
i(\sigma_{s}^{m})(\bar{\tau})  &  =\frac{1}{(2\pi)^{n}}\int_{\check X}
w^{m}e^{-h_{0}^{I_{s}}}\left\Vert dZ\right\Vert _{K}^{1/2}\bar{\tau}_{0}\,
\frac{\omega^{n}}{n!}\\
&  =\int_{\check{P}_X}e^{m\cdot{y}_{s}-({x\cdot y}_{s}-g_{0}-s\psi)}(\det
G_{s})^{1/4} \left(\frac{1}{(2\pi)^{n}}  \int_{\mu^{-1}({x})}e^{im\cdot{\theta}}\bar{\tau}_{0}
d{\theta}\right)  d{x}\\
&  =\int_{\check{P}_X}e^{-s(({x}-m)\cdot\partial\psi/\partial{x}-\psi)}%
e^{-(({x}-m)\cdot{y}_{0}-g_{0})} (\det G_{s})^{1/4}\hat{\bar{\tau}}_{0,m}%
({x})\,d{x.}%
\end{align*}
Let $f_{m}:=({x}-m)\cdot\partial\psi/\partial{x}-\psi.$ Then by an argument similar to that in the proof of Lemma \ref{lemma:L2asymp} and by Lemma
\ref{lemma:fm}, we can use Laplace's approximation with equation
(\ref{eqn:detGslim}) to obtain as $s\rightarrow\infty$ that%
\begin{align*}
i(\sigma_{s}^{m})(\bar{\tau})  &  \sim\left(  \frac{2\pi}{s}\right)
^{n/2}\frac{e^{s\psi(m)+g_{0}(m)}s^{n/4}(\det\operatorname{Hess}\psi
(m))^{1/4}\hat{\bar{\tau}}_{0,m}(m)}{\sqrt{\det\operatorname{Hess}\psi(m)}}\\
&  =(2\pi)^{n/2}s^{-n/4}e^{g_{s}(m)}(\det\operatorname{Hess}_{m}\psi
)^{-1/4}\hat{\bar{\tau}}_{0,m}(m).
\end{align*}
Using Lemma \ref{lemma:L2asymp}, we have
\[
\frac{(\det G_s)^{1/4}}{\left\Vert \sigma_{s}^{m}\right\Vert _{L^{2}}}\sim
s^{n/4}(\det\operatorname{Hess}\psi(m))^{1/4}\pi^{-n/4}e^{-g_{s}%
(m)},\,s\rightarrow\infty
\]
from which we see that as $s\rightarrow\infty$%
\begin{align*}
i\left(  \frac{\sigma_{s}^{m}(\det G_s)^{1/4}}{\left\Vert \sigma_{s}%
^{m}\right\Vert _{L^{2}}}\right)  (\bar{\tau})  &  \sim(2\pi)^{n/2}%
s^{-n/4}e^{g_{s}(m)}(\det\operatorname{Hess}\psi(m))^{-1/4}\hat{\bar{\tau}%
}_{0,m}(m)\\
&  \qquad\qquad\times s^{n/4}(\det\operatorname{Hess}\psi(m))^{1/4}\pi
^{-n/4}e^{-g_{s}(m)}\\
&  =2^{n/2}\pi^{n/4}\hat{\bar{\tau}}_{0,m}(m)
\end{align*}
as desired.
\end{proof}

\bigskip
Observe that the asymptotics of the normalized sections $\sigma_{s}^{m}/
\left\Vert \sigma_{s}^{m}\right\Vert _{L^{2}}$ described by
Theorem \ref{thm:sec_asymps} contain the additional term $(\det G_{s})^{1/4}.$
This extra factor is better understood in the context of the degeneration of ${\cal H}^Q_{I_s}$
which we now study.

To have the spaces $\Gamma(\sqrt{|K_s|})$, for all $s\geq 0$, as subspaces of a given fixed
vector space, we consider $\alpha \in \Gamma(\sqrt{|K_s|})(U)$, for an open set $U\subset X$,
as a map $\alpha: {\cal{X}} (U)^n\to C(U)$, where $\cal{X} (U)$ is the space of smooth complex
vector fields on $U$ and $C(U)$ is the space of continuous complex valued functions on $U$.
Then, we define
$$
\lim_{s\to\infty}\alpha_s = \beta \Leftrightarrow \lim_{s\to\infty} \alpha_s(X_1,\dots,X_n)= \beta (X_1,\dots,X_n),
$$
for all $X_1,\dots X_n \in \cal{X}(U)$, where on the right hand side we consider pointwise convergence in $C(U)$.
An $n$-form $\beta \in \Omega^n(U)\otimes \C$ then also defines a map $\sqrt{|\beta|}: {\cal{X}} (U)^n\to C(U)$, given by
$\sqrt{|\beta|}(X_1,\dots,X_n)=|\beta (X_1,\dots,X_n)|^{\frac12}$, for $X_1,\dots ,X_n\in \cal{X} (U)$.
Consider now the global $n$-form $dX= dx^1\wedge\cdots\wedge dx^n$, vanishing on $\cup_{i=1}^r D_i$,
and the corresponding map $\sqrt{|dX|}$.

\begin{lemma}\label{onu}
In the sense defined above,
\[
\lim_{s\rightarrow\infty}\frac{\mu_{I_{s}}}{\left(  \det G_{s}\right)  ^{1/4}%
}=\sqrt{\left\vert dX\right\vert}.
\]
\end{lemma}

\begin{proof}
On the open orbit, using Lemma \ref{lemma:normdz}, we have
\[
\frac{\mu_{I_{s}}}{(\det G_{s})^{1/4}}=\frac{\sqrt{\left\vert dZ_{s}\right\vert}}{\sqrt{\det G_{s}}}.
\]
But $dz=Gdx+id\theta$, which implies $dZ_{s}\sim s^{n}%
\det\left(  \operatorname{Hess}_{x}\psi\right)  dX$, and $\det G_{s}\sim
s^{n}\det\left(  \operatorname{Hess}_{x}\psi\right)  $, so%
\[
\lim_{s\rightarrow\infty}\frac{\sqrt{\left\vert dZ_{s}\right\vert}}{\sqrt{\det
G_{s}}}=\lim_{s\rightarrow\infty}\frac{s^{n/2}\sqrt{\det\left(
\operatorname{Hess}_{x}\psi\right)  }\sqrt{\left\vert dX\right\vert}}
{s^{n/2}\sqrt{\det\left(  \operatorname{Hess}_{x}\psi\right)  }}%
\]
as desired. Note that on $\cup_{i=1}^r D_i$ both sides vanish.
\end{proof}

\bigskip

We are now ready to explain the factor of $(\det G)^{1/4}$ which appears in
Theorem \ref{thm:sec_asymps}. As we see, the same term appears in the denominator in Lemma \ref{onu}.
Recall the orthogonal basis of ${\cal H}^Q_{I_s}$ given by $\{\hat\sigma^m_s=\sigma^m_s\otimes \mu_{I_s}\}_{m\in P_L\cap \Z^n}$, and also that
$$
||\hat \sigma^m_s||_{L^2} = ||\sigma^m||_{L^2},
$$
since the hermitian structure on $\sqrt{|K_{I_s}|}$ gives $||\mu_{I_s}||=1$.
Combining Lemma \ref{onu} with Theorem \ref{thm:sec_asymps}, we obtain the
following theorem.

\begin{theorem}\label{thmdistrib}
\[
\lim_{s\rightarrow\infty}  \frac{\hat \sigma_{s}^{m}}{\left\Vert \hat \sigma
_{s}^{m}\right\Vert _{L^{2}}}
= \lim_{s\rightarrow\infty}  \frac{ \sigma_{s}^{m}}{\left\Vert \sigma
_{s}^{m}\right\Vert _{L^{2}}} \otimes \mu_{I_s} =
2^{n/2}\pi
^{n/4}\delta^{m}\otimes\sqrt{\left\vert{dX}\right\vert},
\]
in the sense that
\[
\lim_{s\rightarrow\infty}  \frac{i (\sigma_{s}^{m})\otimes \mu_{I_s}}{\left\Vert  \sigma
_{s}^{m}\right\Vert _{L^{2}}}  (\tau; X_1,\dots,X_n)=
2^{n/2}\pi
^{n/4}\delta^{m}(\tau)\left\vert{dX(X_1,\dots,X_n)}\right\vert^{\frac12},
\]
for all test sections $\tau\in\Gamma(\bar L)$ and smooth complex vector fields $X_1,\dots,X_n\in {\cal X}(X)$.
\end{theorem}

\begin{remark}
This Theorem justifies the definition of a natural inner product in ${\cal H}^Q_\R$
defined by declaring $\{2^{n/2}\pi^{n/4}\delta
^{m}\otimes \sqrt{|dX|}\}_{m\in P_{X}\cap \Z^n}$ to be an orthonormal basis.
\end{remark}

We note that the results of Theorem \ref{thmdistrib} are also valid when $X$ is not compact provided
that the growth of $\psi$ at infinity is appropriately controlled, so that the function on $\check P_X$
$$
e^{-c \int_0^1 {^t(x-m)}\cdot G_s(m+t(x-m))\cdot(x-m) dt} (\det G_s)^{\frac12}
$$
is bounded for some sufficiently small $c>0$. This ensures the existence of the $L^2$-norms in question and
also that one can still apply the Laplace approximation to obtain the convergence to Dirac delta distributions.

\section{The BKS pairing}
\subsection{The Blattner--Kostant--Sternberg half-form pairing}

Let $K_{I}$ and $K_{J}$ be the canonical bundles on $X$ associated to two toric
K\"{a}hler complex structures $I$ and $J$.
One may define a nondegenerate sesquilinear pairing $\Gamma\left(
K_{I}\right)  \times\Gamma\left(  K_{J}\right)  \rightarrow C^{\infty}(M)$ by
comparison with the Liouville form. Specifically, for sections $\alpha
\in\Gamma(K_{I})$ and $\beta\in\Gamma(K_{J})$, define the pairing of $\alpha$
and $\beta$ to be the function%
\[
\left\langle \alpha,\beta\right\rangle :=\frac{\alpha\wedge\bar{\beta}%
}{(2i)^{n}(-1)^{\frac{n(n+1)}{2}}\omega^{n}/n!}.
\]

Suppose for the moment that $K_{I}$ and $K_{J}$ admit square roots. Then the
above pairing induces a sesquilinear pairing $\Gamma\left(  \sqrt{K_{I}%
}\right)  \times\Gamma\left(  \sqrt{K_{J}}\right)  \rightarrow C^{\infty}(M)$
via
\[
\left\langle \mu,\nu\right\rangle :=\sqrt{\frac{\mu^{2}\wedge\bar{\nu}^{2}%
}{(2i)^{n}(-1)^{\frac{n(n+1)}{2}}\omega^{n}/n!}},
\]
for sections $\mu\in\Gamma\left(  \sqrt{K_{I}}\right)  $ and $\nu\in
\Gamma\left(  \sqrt{K_{J}}\right)  .$ Note that when $I=J$, the half-form
pairing above reduces to the canonical hermitian structure
(\ref{eqn:rootKherm}) on $\sqrt{K_{I}}$. Moreover, this pairing of half-forms
induces a pairing on the half-form corrected prequantizations of $X$ which is
known as the Blattner--Kostant--Sternberg (BKS) pairing:
\[
\left\langle s\otimes\mu,t\otimes\nu\right\rangle _{BKS}:=\int_{X}h^{\ell
}(s,t)\left\langle \mu,\nu\right\rangle \frac{\omega^{n}}{n!}.
\]
Let $\ell,L$ be the line bundles over $X$ as in Section \ref{sec:corrQ}.
The fact that $\ell \otimes \sqrt{K} \cong L \otimes \sqrt{|K|}$  will motivate
the definition of the BKS pairing even in the case when $\sqrt{K}$ does not exist.

Let us examine the half-form pairing on $X$ in a little more detail. A
short computation shows that on the open orbit
\begin{equation}
\left\langle dZ_{I},dZ_{J}\right\rangle =\det\left(  \frac{G_{I}+G_{J}}%
{2}\right)  >0, \label{eqn:Kpairing}
\end{equation}
since both $G_{I}$ and $G_{J}$ are positive definite. Notice that although in
general one expects $\langle dZ_{I}, d{Z}_{J}\rangle$ to be complex, in the toric
case it turns out to be real (and positive) since the unitarization of $K_{I}$ is equal to the
unitarization of $K_{J}$, whence the phases do not contribute to the pairing. To
put it in another way, the pairing of $K_{I}$ and $K_{J}$ is entirely captured by
the modulus bundles $\left\vert K_{I}\right\vert $ and $\left\vert
K_{J}\right\vert $.
Consequently we define a pairing between sections of $\sqrt{|K_I|}$ and $\sqrt{|K_J|}$ by
\begin{equation}
\left\langle \mu_{I},\mu_{J}\right\rangle :=\frac{\langle
dZ_{I}, d{Z}_{J}\rangle^{\frac12}}{\left\Vert {dZ_{I}}\right\Vert^{\frac12}
_{K_I}\left\Vert {dZ_{J}}\right\Vert_{K_J}^{\frac12}} \label{eqn:normKBKS},
\end{equation}
where $\mu_I, \mu_J$ are defined in (\ref{mumu}). Then,
in the general case when $\sqrt{K}$ may not exist, we \textit{define} a
BKS pairing by
\begin{align}
\left\langle \hat \sigma_{I},\hat\sigma_{J}
\right\rangle  &  := \frac{1}{(2\pi)^n}\int_{X}h^{L}(\sigma_{I},\sigma
_{J})\left\langle \mu_{I},\mu_{J}\right\rangle
\frac{\omega^{n}}{n!}, \label{eqn:BKSonL}
\end{align}
where $\hat \sigma_I = \sigma_I
\otimes \mu_I \in {\cal H}^Q_I, \hat \sigma_J = \sigma_J \otimes \mu_J\in {\cal H}^Q_J.$

This pairing coincides with the inner product (\ref{innerproduct}) in ${\cal H}^Q_I$ when $I=J$.
From (\ref{eqn:normKBKS}), we also see that it coincides with the standard BKS pairing
in the case when $\sqrt{K}$ exists.

Let $h_m^I=(x-m)\frac{\partial g_I}{\partial x}-g_I$, where $g_I$ is the symplectic potential
defining the complex structure $I$. From Theorem \ref{cor:holsections} and (\ref{eqn:Kpairing})
it is straightforward to obtain

\begin{proposition}
\label{prop:BKS}
In the orthogonal basis $\{\hat \sigma^m_I\}_{m\in P_L\cap \Z^n}$ for ${\cal H}^Q_I$ and
$\{\hat \sigma^m_J\}_{m\in P_L\cap \Z^n}$ for ${\cal H}^Q_J$, we have
\begin{equation}
\langle \hat \sigma_I^m, \hat \sigma_J^{m'} \rangle =
\delta_{mm'} \int_{P_X} e^{-h_m^I-h_m^J} \sqrt{\det\left(  \frac{G_{I}+G_{J}}{2}\right)} dx,
\end{equation}
for $m,m'\in P_L\cap \Z^n$.
\end{proposition}

\subsection{Unitarity}

Let $I_{s}$ denote the \textquotedblleft simple\textquotedblright\ family of
complex structures on $X$ associated to the symplectic potentials $g_{s}%
=g_{P}+\varphi+s\psi$ for $s\in\lbrack0,\infty)$. (Recall that we could have
more general deformations of $I_{0}$ defined by $\psi(s)$.) We can consider
the BKS pairing for two values $s,s^{\prime}$ in the same simple family. As we
will see, even for these simple families, the BKS pairing is not unitary.

From Proposition \ref{prop:BKS}, we see that $I$- and $J$-holomorphic
sections associated with different integral points are orthogonal. In the
following we will therefore consider only one \textquotedblleft
Fourier\textquotedblright\ sector at a time, that is, a one-dimensional
subspace of the quantization space. Let $m\in P_{L}\cap \Z^n$, and consider the
corresponding monomial section $\hat \sigma_{s}^{m}\in {\cal H}^Q_{I_s}$.

Since the BKS pairing $\langle\hat \sigma_{s}^{m},\hat \sigma_{s^{\prime}}^{m}
\rangle_{BKS}$ is real and positive (the integrand is positive), the unitarity
of the BKS pairing map for the complex structures $s,s^{\prime}$ is equivalent
to
\[
\langle\hat \sigma_{s}^{m},\hat \sigma_{s^{\prime}}^{m}\rangle_{BKS}=||\hat \sigma_{s}%
^{m}||_{L^2}\cdot||\hat \sigma_{s^{\prime}}^{m}||_{L^2}.
\]
For our choice of monomial sections, writing out the integral in Proposition
\ref{prop:BKS} shows that for some function $\alpha\in C^\infty(\R_{\geq 0})$,
\[
\langle\hat \sigma_{s}^{m},\hat\sigma_{s^{\prime}}^{m}\rangle_{BKS}=\alpha(s+s^{\prime
})>0.
\]
Unitarity then implies $\alpha(s+s^{\prime})=\sqrt{\alpha(2s)\alpha
(2s^{\prime})}$. Putting $\alpha=e^{f}$, we get
\begin{equation}
f\left(  \frac{s+s^{\prime}}{2}\right)  =\frac{f(s)+f(s^{\prime})}{2},
\label{meio}
\end{equation}
and differentiating in $s$ we obtain
\[
\frac{1}{2}f^{\prime}\left(  \frac{s+s^{\prime}}{2}\right)  =\frac{1}%
{2}f^{\prime}(s).
\]
Therefore,

\begin{lemma}
The BKS pairing between the $I_{s}$- and $I_{s^{\prime}}$-holomorphic
quantizations of $X$ is unitary if and only if for each $m\in P_{L}\cap \Z^n$ and
for each $s\geq0$,
\begin{equation}
\left\Vert \hat \sigma_{s}^{m}\right\Vert_{L^2} =\left\Vert \hat \sigma_{0}^{m}\right\Vert_{L^2}
e^{sb}, \label{eqn:BKSunitarity-cond}%
\end{equation}
for some constant $b$.
\end{lemma}

Comparing with Lemma \ref{lemma:L2asymp}, we see that if
(\ref{eqn:BKSunitarity-cond}) holds, then we have $b=-2\psi(m)$ and
$\left\Vert \hat \sigma_{0}^{m}\right\Vert_{L^2} =\pi^{n/2}e^{2g_{0}(m)}.$ Moreover,
replacing $\psi$ by $\psi+const$ \ does not change the complex structure.
Hence, we can assume that $\psi(m)=0$ so that the BKS pairing is unitary if
and only if $\left\Vert \hat \sigma_{s}^{m}\right\Vert_{L^2} =\left\Vert \hat \sigma_{0}%
^{m}\right\Vert_{L^2}$, that is, that the $L^{2}$-norm of $\hat \sigma_{s}^{m}$ is
independent of $s$.

 By the argument above, the following theorem implies that the BKS pairing is
not unitary along the simple family $I_s, s\geq 0$.

\begin{theorem}
\label{thm:norm-not-const}For sufficiently large $s$,
\[
\frac{d}{ds}\left\Vert \hat \sigma_{s}^{m}\right\Vert_{L^2} ^{2}\neq0.
\]
In particular, $\left\Vert \hat \sigma_{s}^{m}\right\Vert_{L^2} $ is not constant, whence
the BKS pairing is not unitary.
\end{theorem}

\begin{proof}
We have, using
the identity $\left.  \frac{d}{ds}\right\vert _{s=0}\det(A+sX)=\det
(A)tr(A^{-1}X),$
$$
\frac{d}{ds}\left\Vert \hat\sigma_{s}^{m}\right\Vert ^{2}_{L^2} =
 \int_{{P_X}}\left(  2\left(
\psi-({x}-m)\cdot\frac{\partial\psi}{\partial x}\right)  +\frac{1}{2}tr
\left(G_s^{-1}\operatorname{Hess}\,\psi  \right)\right)  e^{-2h_m^{I_s}} \sqrt{\det G_s}\,d{x.}
$$
Without loss of generality, assume that $\psi$ is scaled so that $\psi(m)=0$.
Then by Lemma \ref{lemma:fm}, $\left(  \psi-({x}-m)\cdot\frac{\partial\psi
}{\partial x}\right)  $ is strictly negative, so that for large $s $ the
integrand is also strictly negative, since $tr
\left(G_s^{-1}\operatorname{Hess}\,\psi  \right)$ converges pointwise to zero as $s\to\infty$,
which implies that $\frac{d}{ds}\left\Vert
\hat \sigma_{s}^{m}\right\Vert ^{2}_{L^2}$ is not equal to zero.
\end{proof}

\subsection{Quantum connection\label{sec:connection}}

Let ${\cal I}$ be the set of compatible toric complex structures on $X$. This
naturally corresponds to the space of allowed Guillemin-Abreu symplectic potentials modulo
additive constants, since two symplectic potentials define the same holomorphic coordinates
iff they differ by an additive constant. Note that if $g_I$ is a symplectic potential
defining the complex structure $I$, then $g_I+s\psi$ where $s\in \R, \psi\in C^\infty (P_X)$
will also be an allowed symplectic potential provided that $|s|$ is sufficiently small.
Therefore, we can regard ${\cal I}$ as an open subset of the affine space
$g_{P_X}+(C^{\infty}(P_X))/\R$. Fix a point $p\in {\check P}_X$. In the following we will
assume that for each complex structure $I\in {\cal I}$ a symplectic potential $g_I$
was chosen such that $g_I(p)-g_{P_X}(p)=0$.

One can define a vector bundle over ${\cal I}$, $\mathcal{H}^Q\rightarrow\mathcal{I}$, with fiber
${\cal H}_{I}^{Q}$.
The bundle $\mathcal{H}^Q$ is the
natural setting for studying the dependence of K\"{a}hler quantization on the
choice of complex structure.

As mentioned in the introduction, Axelrod, Della Pietra and Witten in
\cite{Axelrod-DellaPietra-Witten} (see also the related work \cite{Hitchin90} of Hitchin),
for the case where $\mathcal{I}$ is the space of linear complex structures on
$\mathbb{R}^{2n}$ which are compatible with the standard symplectic form,
but without including the half-form correction,
introduced a natural unitary connection, which is called
the \emph{quantum connection}, on the analogue of the quantum bundle $\mathcal{H}^Q$.
This quantum connection is defined to be the
projection of the trivial connection in $\mathcal{I}\times\Gamma(\ell)$ to
$\mathcal{H}^Q$.

In our case, the BKS pairing induces a
connection, $\nabla^Q$, on $\mathcal{H}^Q\rightarrow\mathcal{I}$ as follows. Recall that different Fourier
modes are orthogonal with respect to the BKS pairing, so that we can consider each Fourier mode labeled by $m\in P_L\cap\Z^n$
separately.
Consider the monomial sections $\hat\sigma^m_I\in {\cal H}^q_I$, $I\in {\cal I}$, $m\in P_L\cap\Z^n$. Define
$\nabla^Q_\psi$ by
$$
\langle \nabla^Q_\psi \hat\sigma^m_I, \hat \sigma^m_I\rangle = \frac{d}{ds}_{|_{s=0}}
\langle\hat\sigma^m_{I+s\psi}, \hat\sigma^m_I\rangle,
$$
where $\psi\in C^\infty (P_X)$ with $\psi(p)=0$ and where
$I+s\psi$ denotes the complex structure defined by
the symplectic potential $g_I+s\psi$ for $|s|$ sufficiently small. Note
that, by construction, $\nabla^{Q}$ is unitary.

A final implication of the fact that all of the complex structures arising
from points in $\mathcal{I}$ have the same \textquotedblleft
phases\textquotedblright\ (that is, of Lemma \ref{lemma:I-indep}) and of the
consequent reality of the BKS pairing is the following theorem.

\begin{theorem}
The global frame for ${\cal H}^Q\to {\cal I}$ given
by $\{\hat \sigma^m_I / ||\hat \sigma^m_I||\}_{m\in P_L\cap \Z^n}$
is horizontal with respect to $\nabla^Q$. Therefore, the connection $\nabla^Q$ is flat.
\end{theorem}

\begin{proof}
We have, from the reality of the BKS pairing, for each Fourier mode $m\in P_L\cap\Z^n$ and for
any $\psi\in C^\infty(P_X)$ with $\psi(p)=0$,
$$
0= \frac{d}{ds}_{|_{s=0}} \langle \frac{\hat\sigma_{I+s\psi}^m}{||\hat \sigma^m_{I+s\psi}||},
\frac{\hat\sigma_{I+s\psi}^m}{||\hat \sigma^m_{I+s\psi}||}\rangle = 2 \langle
\nabla^Q_\psi \frac{\hat\sigma_{I}^m}{||\hat \sigma^m_I||}, \frac{\hat\sigma_{I}^m}
{||\hat \sigma^m_I||}\rangle,$$
whence $\nabla^Q_\psi \frac{\hat\sigma_{I}^m}{||\hat \sigma^m_I||}=0$.
\end{proof}

\section{Appendix: The square root of $K$}

For toric varieties, the following proposition explains how the existence of
$\sqrt{K}$ depends on the combinatorics of the fan $\Sigma$ associated to $X$. Note also that, in
general, if a square root of the canonical bundle exists then there may be
many choices of square root, and they are parameterized by $H^{1}(X,\mathbb{Z}_{2})$.
For toric varieties, $H^{1}(X,\mathbb{Z}_{2})=\{0\},$ and
so if it exists, $\sqrt{K}$ is unique.

Let $\{\nu_j\}_{j=1,\dots,r}$ be primitive generators of the 1-dimensional cones in the fan $\Sigma$. Let $L$ be
a holomorphic line bundle on $X$. The divisor of a torus-invariant meromorphic
section of $L$ on any holomorphic vertex
chart $U_v$ determines the section up to multiplicative constant. Indeed,
a torus-invariant principal divisor is of the form
$$
\sum_{i=1}^r \langle \alpha, \nu_j\rangle D_j,
$$
for $\alpha \in \R^n.$
If, say, $\nu_1,\dots,\nu_n$ are the generators associated to $U_v$, which form a basis of $\Z^n$, then the restriction
of the principal divisor to $U_v$  determines it completely since $\alpha$ is fixed by its inner products with
$\nu_1,\dots,\nu_n$.
Such a line bundle has a system of meromorphic locally trivializing sections $\{\mathbf{1}_v^L\}$ on the
holomorphic vertex charts, as in
Section \ref{subsec:line-bundles}, defined uniquely (up to constants) by the property that
$\operatorname{div}(\mathbf{1}_v^L)_{\vert_{U_v}}=0$.

If $c_1(X)$ is even then $\sqrt{K}$ exists and it has a system of holomorphic trivializations
on the holomorphic vertex charts given by $\{(U_v,\mathbf{1}_v)\}$, where $\mathbf{1}_v:= \mathbf{1}_v^{\sqrt{K}}$.
Then, $\{(U_v,\mathbf{1}_v\otimes \mathbf{1}_v)\}$ gives
a system of trivializing sections for $K$ and we have, up to an irrelevant constant,
$\mathbf{1}_v\otimes\mathbf{1}_v =dW_v$. Therefore, the sections $dW_v$ have even divisors and
we can write $\mathbf{1}_v= \sqrt{dW_v}$ where
$\{(U_v,\sqrt{dW_v})\}$ is a system of holomorphic trivializing sections on the holomorphic vertex charts for $\sqrt{K}$. We therefore have

\begin{proposition}\label{evendiv}
If $c_1(X)$ is even, then the sections $dW_v$ have even divisors and $\{(U_v,\sqrt{dW_v})\}$ is a system of holomorphic
trivializations of $\sqrt{K}.$
\end{proposition}

As a consequence, we obtain the following useful criterion for the existence of $\sqrt{K}.$
For a vertex $v$, let us call vertex basis associated to $v$ to a basis of $\Z^n$ given by the primitive generators
of the 1-dimensional cones defining $v$.

\begin{proposition}
$K$ admits a square root if and only if for each 1-dimensional cone $F$ in $\Sigma$
the sum of the coordinates of the primitive generator $\nu_{F}$ expressed in any one of the vertex basis is odd.
\end{proposition}

\begin{proof}
We have seen that the existence of $\sqrt{K}$ is equivalent to having a system of trivializing sections
$\{\sqrt{dW_v}\}$ where the $dW_v$ all have even divisors. Choose a vertex $v$ and let $\nu_1,\dots\nu_n$ be a
vertex basis for $v$. Then, using Lemma \ref{lemmacox} and the fact that $dZ_v$ is a constant multiple of $dZ$,
$$
\operatorname{div}(dW_{v}) = \operatorname{div}(w_{v}^{\mathbf{1}})+ \operatorname{div} (dZ_{v}) =
\sum_{i=n+1}^r \left(\sum_{j=1}^n \nu_i^j \right) D_i - \sum_{i=n+1}^r D_i,
$$
where $\nu_i^j$ is the $jth$ coordinate of the vector $\nu_i$ in the vertex basis. Therefore, this divisor is even iff
$\sum_{j=1}^n \nu_i^j$ is odd for all $i$. Clearly, this happens for any $v$ and the proposition follows.
\end{proof}

\bigskip\bigskip

\textbf{{\large Acknowledgements:}} We wish to thank Thomas Baier for
extensive discussions on the subject of the paper. We also wish to thank the referee
for several helpful comments which helped improving the clarity of the text. The authors are partially
supported by the Center for Mathematical Analysis, Geometry and Dynamical
Systems, IST, through the Funda\c{c}\~{a}o para a Ci\^{e}ncia e a Tecnologia
(FCT / Portugal).

\bigskip


\end{document}